\documentclass[11pt]{amsart}

\usepackage[dvips]{color}

\usepackage{comment}

\usepackage{enumerate}
\usepackage{amssymb}
\usepackage{epsfig}
\usepackage{amsmath}
\usepackage{color}
\usepackage{mathrsfs}
\newtheorem{theorem}{Theorem}
\newtheorem{lemma}[theorem]{Lemma}

\newtheorem{blank}[theorem]{}

\newtheorem{example}[theorem]{Example}

\newcommand{\al}{{\alpha}}

\title{Limits of Tangents of Surfaces}

\author{Jo\~ao Cabral}
\address{
 Centro de Matem\'atica e Aplica\c{c}\~{o}es (CMA), FCT, UNL and Departamento de Matem\'atica, FCT, UNL, 2829-516 Caparica, Portugal}
\email{jpbc@fct.unl.pt} 

\author{Orlando Neto} 

\address{
 CMAF, Faculdade de Ci\^{e}ncias,
 Universidade de Lisboa, Campo Grande 1749-016, Lisboa, Portugal}
\email{orlando60@gmail.com}

\date{\today}
\begin{document}
\maketitle

\begin{abstract} 
{We compute the limit of tangents of an arbitrary surface. 
We obtain as a byproduct an embedded version of Jung's desingularization theorem for surface singularities with finite limits of tangents.}
\end{abstract}

\section{Introduction}

Let $S$ be a singular surface of the germ of a complex analytical manifold $M$ at a point $o$.
Theorem 2.3.7 of \cite{LE} states that the limit of tangents $\Sigma_o(S)$ 
is the union of the dual of the tangent cone $C_o(S)$ of $S$ with a finite set of projective lines 
of $\mathbb P_o^*M$. 

Theorem 1.4.4.1 of \cite{AMERICAN} gives a criteria to decide 
if a certain projective line is contained or not in $\Sigma_o(S)$,
provided a non  degeneracy condition is verified. The proof relies on a commutative diagram
\[
\begin{array}{ccccc}
  \Gamma  & &\leftarrow &  & \widetilde\Gamma  \\
  \downarrow & & & &   \downarrow \\
S  & & \leftarrow &  & \widetilde S
\end{array}
\]
where $\widetilde S$ is the strict transform of $S$ by the blow up $\pi:\widetilde M\to M$ of $M$ at $o$,
and $\Gamma$ is the conormal of $S$. 
Hence $S\subset M$, $\widetilde S\subset \widetilde M$ and $\Gamma\subset\mathbb P^*M$:
The surface $\widetilde\Gamma$ has a more enigmatic status.
If we had an immersion of $\widetilde\Gamma$ into a manifold $\mathcal X$ endowed with a symplectic structure
we could iterate the process, blowing up $\widetilde M$ 
and eventually lifting the need for the non degeneracy condition.
Unfortunately there can be no such symplectic structure on $\mathcal X$.

Theorem \ref{CLASSICAL} presents a new proof of Theorem 2.3.7 of \cite{LE}. 
It sheds a new light on the problem. 
Set $E=\pi^{-1}(o)$. 
There is a vector bundle $T^*\langle \widetilde M/E\rangle$ on $\widetilde M$ with sheaf of sections 
the locally free $\mathcal O_{\widetilde M}$-module of
logarithmic differential forms on $\widetilde M$ with poles along $E$.
Let $\pi_{\widetilde M}:\mathbb P^*\langle \widetilde M/E\rangle\to \widetilde M$ be the associated projective bundle.
There is a commutative diagram
\begin{equation}\label{FIRSTBU}
\begin{array}{ccccc}
  \mathbb P^*M  & \overset{\;\;\widehat\pi}{\longleftarrow} &   \mathbb P^*\langle \widetilde M/E\rangle   \\
  \pi_M\downarrow \qquad &  &   \qquad\downarrow \pi_{\widetilde M}\\
  M  & \overset{\;\;\pi}{\longleftarrow} &    \widetilde M
\end{array}
\end{equation}
where $\widehat\pi$ is blow up of $\mathbb P^*M$ along $\pi_M^{-1}(o)$.
The projectivization $\mathbb P^*\langle \widetilde M/E\rangle$ of the vector bundle  $T^*\langle \widetilde M/E\rangle$
provide an ambient space for $\widetilde\Gamma$, which is a \em Legendrian variety \em of 
$\mathbb P^*\langle \widetilde M/E\rangle $, the strict transform of
$\Gamma$ by $\widehat\pi$ and the \em conormal \em of $\widetilde S$. 

There is a canonical embedding of the projective cotangent bundle $\mathbb P^*E$
into $\mathbb P^*\langle \widetilde M/E\rangle$. Moreover, 
$\widetilde\Gamma'=\widetilde\Gamma\cap \pi_{\widetilde M}^{-1}(E)$ is contained
in $\mathbb P^*E$ and $\widetilde\Gamma'$ is a Legendrian variety of $\mathbb P^*E$. 
Since $\pi_{\widetilde M}(\Gamma')=\widetilde S\cap E$, there are $\sigma_1,...,\sigma_n\in \widetilde S\cap E$ such that
\[
\widetilde\Gamma'=\mathbb P^*_{\widetilde S\cap E}E\cup\cup_i \mathbb P^*_{\sigma_i}E.
\]
Here $\mathbb P^*_{\widetilde S\cap E}E$ denotes the \em conormal \em of the curve $\widetilde S\cap E$, the smallest Legendrian curve 
of $\mathbb P^*E$ that projects on $\widetilde S\cap E$, and 
$ \mathbb P^*_{\sigma_i}E$ denotes the fiber at $\sigma_i$ of $\mathbb P^*E$.
The restriction of $\widehat\pi$ to $\mathbb P^*E$ defines a map 
\[
\widehat\pi:\mathbb P^*E\to \mathbb P_o^*M.
\]
Moreover,
\[
\Sigma_o(S)=\widehat\pi(\mathbb P^*_{\widetilde S\cap E}E\cup\cup_i \mathbb P^*_{\sigma_i}E).
\]
More precisely,
the dual curve of $\widetilde S\cap E$ is the image of the Legendrian curve $\mathbb P^*_{\widetilde S\cap E}E$
and the pencils of planes are the images of the projective lines $\mathbb P^*_{\sigma_i}E$.

There is a natural logarithmic generalization of the notion of limit of tangents.
The surface $\widetilde\Gamma$ is the \emph{conormal} of the surface $\widetilde S$, in a sense that will be precised in section \ref{LCG}.
We call $\Sigma^E_\sigma(\widetilde S)=\widetilde\Gamma\cap \pi_{\widetilde M}^{-1}(\sigma)$ the \em logarithmic limit of tangents of $\widetilde S$ at $\sigma$, with poles along $E$. \em 

We reduce in this way the computation of the limit of tangents of a surface to the problem of 
deciding if, given $\sigma\in \widetilde S\cap E$, $\Sigma^E_\sigma(\widetilde S)=\mathbb P^*_\sigma E$ or $\Sigma^E_\sigma(\widetilde S)$ is finite.
This problem is solved by Theorem \ref{DECIDE}.

The introduction of the notion of logarithmic limit of tangents allows us to iterate the construction
that gives us a new proof of Theorem 2.3.7 of \cite{LE}. 
In order to compute the limit of tangents we need to introduce a canonical process of reduction of singularities
for surfaces. 
In general this process will terminate before we desingularize the surface $S$, 
giving us enough information to compute the limit of tangents of $S$. 
If the limit of tangents of $S$ is finite, 
the process will terminate when all singular points of some strict transform of $S$ are quasi ordinary.
See Theorem \ref{JUNG}. We obtain in this way an embedded version of Jung's desingularization algorithm.

Let $C$ be the singular locus of $S$.
Roughly speaking, the algorithm of reduction of singularities proceeds in the following way:
\begin{enumerate}[$(a)$]
\item
We blow up $M$ at $o$;
\item
Given a regular point $\sigma$ of the inverse image $N$ of $o$ by the sequence of
blow ups, we blow up $\sigma$ if $\sigma$ is a singular point of the strict transform $\widetilde C$
of $C$ or $\widetilde C$ is not transversal to $N$ at $\sigma$.
\item
If a strict transform $\widetilde S$ of $S$ by the sequence of blow ups contains a connected component $W$ of the 
singular locus of $N$, we blow up $W$.
\item
If $\sigma$ is an isolated point of $\widetilde S\cap W$, there is a local plane projection $\rho$ such that
$\rho^{-1}(\rho(N))=N$. If the germ at $\rho(\sigma)$ of the discriminant $\Delta_\rho S$ is not contained in $\rho(N)$,
we blow up $W$.
\end{enumerate}
In cases $(c),(d)$ this algorithm works in a very similar way to Jung's algorithm.
In case $(b)$ the singular locus of $S$ takes the place of the discriminant of $S$.

In section \ref{LCG} we introduce some basic notions of Logarithmic Contact Geometry.
In section \ref{LLT} we introduce the notion of 
logarithmic limits of tangents and compute logarithmic limits of tangents of  quasi-ordinary surfaces. 
The fact that a surface $S$ is quasi-ordinary relative to a given projection does not mean 
that it is quasi ordinary relatively to another projection.
The situations changes if the limit of tangents of the surface is finite.
This fact is a key argument in the proofs of Theorems \ref{CHANGESMOOTH} and \ref{CHANGETRANS}.
These Theorems relate the logarithmic limit of tangents of a surface with its discriminant,
in the spirit of one of the statements of Theorem 1.4.4.1 of \cite{AMERICAN}.
Section \ref{FSB} studies the logarithmic limits of tangents $\Sigma_\sigma^N(S)$ 
when $\sigma$ is a singular normal crossings point of $N$.
Here the main tool is the first sequence of blow ups.
It reduces the computation of $\Sigma_\sigma^N(S)$ when $\sigma$ is a singular point of $N$
to the computation of $\Sigma_\sigma^N(S)$ when $\sigma$ is a regular point of $N$.
In section \ref{NDC} we study the behaviour of $\Sigma_\sigma^N(S)$ by blow up when 
the non degeneracy condition \em $C_\sigma(S)$ does not contain $C_\sigma(N)$ \em is verified.

In section \ref{SSB} we study the behaviour of $\Sigma_\sigma^N(S)$ by blow up without assuming the non degeneracy condition.
This is the longest and the more technical section of the paper.
In Differential Geometry it is sometimes unavoidable the use of long computations involving systems of local coordinates.
Here the main tool is the second sequence of blow ups, that will be the building block of the reduction of singularities procedure that decides if $\Sigma_\sigma^N(S)=\Sigma_\sigma^N$, when $\sigma$ is a regular point of $N$. 

In section \ref{MRE} we state the main results and use them to compute several limits of tangents.

The second named author would like to thank Bernard Teissier and Le Dung Trang, who showed him the beauty of Singularity Theory.

\section{Logarithmic Contact Geometry}\label{LCG}

All manifolds considered in this paper are complex analytic manifolds.
Let $N$ be a normal crossings divisor of a  manifold $M$.
We will denote by $\Omega_M^1\langle N\rangle$ the \em sheaf of 
 logarithmic differential forms on $M$ with poles along $N$. \em
If $N=\emptyset$, $\Omega_M^1\langle N\rangle$ 
 equals the sheaf $\Omega_M^1$ of differential forms on $M$.
 
 We will denote by $T^*\langle M/N\rangle$ the vector bundle with sheaf of sections $\Omega_M^1\langle N\rangle$.
 We call $\pi_M:T^*\langle M/N\rangle\to M$ the \em logarithmic cotangent bundle of $M$ with poles along $N$. \em
If $N=\emptyset$, $T^*\langle M/N\rangle$ 
 equals the cotangent bundle $T^* M$ of  $M$.
 
 There is a canonical logarithmic $1$-form $\theta_N$ on $T^*\langle M/N\rangle$ that coincides with the canonical $1$-form 
 $\theta$ of $T^*M$ outside of $\pi_M^{-1}(N)$.
 
 We call the projectivization $\pi_M:\mathbb P^*\langle M/N\rangle\to M$ of $T^*\langle M/N\rangle$ the
\em logarithmic projective cotangent bundle of $M$ with poles along $N$. \em
 
 \begin{example}\em
 Assume $(x_1,...,x_n)$ is a system of coordinates on a open set $U$ of $M$ such that $N\cap U=\{x_1\cdots x_\nu=0\}$.
 Given a differential form $\omega$ on $U$ there are holomorphic functions $a_1,...,a_n\in\mathcal O_M(U)$ such that
 \[
 \omega= \sum_{i=1}^\nu a_i\frac{dx_i}{x_i}+\sum_{i=\nu+1}^n a_i{dx_i}.
 \]
 There are $\xi_1,...,\xi_n\in\mathcal O_{T^*\langle M/N\rangle}(\pi_M^{-1}(U))$ such that
 \[
  \theta_N= \sum_{i=1}^\nu \xi_i\frac{dx_i}{x_i}+\sum_{i=\nu+1}^n \xi_i{dx_i}.
 \]
 \end{example}

Assume $\textrm{dim}\: M=n$. Let $\Gamma$ be an analytic subset of dimension $n$ of $T^*M$. 
We say that $\Gamma$ is \em conic \em if $\Gamma$ 
is invariant by the action of $\mathbb C^*$ on the fibers of $T^*M$.
We say that $\Gamma$ is a \em Lagrangian variety \em if the symplectic form
$d\theta$ of $T^*M$ vanishes on the regular part of $\Gamma$.
Notice that $\Gamma$ is a conic Lagrangian variety if and only if
$\theta$  vanishes on the regular part of $\Gamma$. 

We identify $M$ with the graph of the zero section of $T^*M$.
There is a canonical map $\gamma:T^*M\setminus M\to\mathbb P^*M$.
We say that an analytic subset $\Gamma$ of $\mathbb P^*M$ is a 
\emph{Legendrian variety} of $\mathbb P^*M$ if $\gamma^{-1}(\Gamma)$ is a (conic) Lagrangian
variety of $T^*M$.

Let $S$ be a closed irreducible analytic subset of $M$. 
We call \em conormal of $S$ \em 
to the smallest Legendrian variety $\Gamma$ of $\mathbb P^*M$ such that $\pi_M(\Gamma)=S$.
We will denote it by $\mathbb P_S^*M$. If $S$ has irreducible components $S_i$, $i\in I$, we set
$\mathbb P_S^*M=\cup_i\mathbb P_{S_i}^*M$.

Let $\Gamma$ be a closed analytic subset of $\mathbb P^*\langle M/N\rangle$.
Set $\Gamma'=\Gamma\cap\mathbb P^*(M\setminus N)$. 
We say that $\Gamma$ is a \emph{Legendrian variety}  of $\mathbb P^*\langle M/N\rangle$
if $\Gamma'$ is a Legendrian variety of $\mathbb P^*( M\setminus N)$ and $\Gamma$ is the closure of $\Gamma'$.

Let $S$ be an analytic subset of $M$ such that $S$ equals the closure of $S\setminus N$.
We call \emph{conormal} of $S$ to the closure $\mathbb P^*_S\langle M/N\rangle$
of $\mathbb P^*_{S\setminus N}( M\setminus N)$

Given an irreducible Legendrian variety $\Gamma$ of $\mathbb P^*\langle M/N\rangle$,
$\Gamma=\mathbb P^*_S\langle M/N\rangle$, where $S=\pi_M(\Gamma)$. 
The proof follows the arguments of the equivalent proof in the classical case.

\begin{lemma}\label{RESIDUAL}
\emph{(see \cite{ON1})}
Assume $N$ smooth.
Then there is a canonical immersion of $\mathbb P^*N$ into $\mathbb P^*\langle M/N\rangle$.
If $\Gamma$ is a Legendrian variety of $\mathbb P^*\langle M/N\rangle$, then
\[
\Gamma_0=\Gamma\cap\pi_M^{-1}(N)\subset \mathbb P^*N.
\]
Moreover, $\Gamma_0$ is a Legendrian variety of $\mathbb P^*N$.
\end{lemma}

Section \ref{FSB} studies the intersection of $\Gamma$ with $\pi_M^{-1}(N)$ when $N$ is a singular normal crossings divisor.

\section{Logarithmic Limits of Tangents}\label{LLT}

Let $M$ be a germ of a complex manifold of dimension $3$ at a point $o$.
Let $N$ be a normal crossings divisor of $M$.
Let $S$ be a surface of $M$.
Let 
\[
\Sigma^N_o(S)=\mathbb{P}^*_S\langle M/N \rangle\cap \mathbb{P}^*_o\langle M/N\rangle
\]
 be the \em logarithimc limit of tangents of $S$ along $N$ at the point $o$. \em 
 
 If $N$ is empty we get the usual definition of limit of tangents of $S$ at $o$.
If $N$ is smooth, set 
\[
\Sigma^N_o=\mathbb{P}^*_oN\subset \mathbb{P}^*_o\langle M/N\rangle.
\]

\begin{lemma}\label{CONTIDO}
If $N$ is smooth, $\Sigma^N_o(S)\subset\Sigma^N_o$.
\end{lemma}

\begin{proof}
It follows from Lemma \ref{RESIDUAL}.
\end{proof}

Let $\rho$ be a submersion of $M$ into a smooth surface $X$.
We say that $\rho$ is \em compatible \em with $N$ if $\rho^{-1}(\rho(N))=N$.

Let $\Xi_\rho(S)$ be  the \em apparent contour \em of $S$ relatively to the projection $\rho$. 
Let $\Delta_\rho(S)=\rho(\Xi_\rho(S))$ be the \em discriminant \em of $S$ relatively to $\rho$.

Assume $\rho$ is compatible with $N$.
We call the closure $\Xi^N_\rho(S)$ of $\Xi_\rho(S)\setminus N$ the \em logarithmic apparent contour \em  along $N$ of $S$ relatively to $\rho$.

We call \em logarithmic singular set of \em $S$ \em with respect to \em $N$ to the closure $\textrm{Sing}^N(S)$ of 
${\textrm{Sing}(S)\setminus N}$, where $\textrm{Sing}(S)$ is the singular locus of $S$.

We will fix systems of local coordinates $(x,y,z)$, $[(x,y)]$ on $M$ $[X]$ such that
 $o=(0,0,0)$ and $\rho(x,y,z)=(x,y)$. We set $\Delta_z(S)=\Delta_{\rho}(S)$.

\begin{lemma}\label{TRIVIAL}
If $\rho$ is compatible with $N$ and $\Delta_\rho(S)\subset \rho(N)$, 
\[
\Sigma^N_o(S)=\{(0:0:1)\}.
\]
\end{lemma}

\begin{proof}
We can assume $N=\{x=0\}$ or $N=\{xy=0\}$.

Consider the first case. 
The surface $S$ admits a parametrization 
\begin{equation}\label{PARAMETR}
x=t^k,
\qquad
z=t^n\varphi(t,y),
\end{equation}
where $\varphi\in \mathbb C\{t,y\}$ and $\varphi(0,0)\neq 0$.
Replacing (\ref{PARAMETR}) in 
\begin{equation}\label{XIDXX}
\xi\frac{dx}{x}+\eta dy+\zeta dz
\end{equation}
we conclude that $\mathbb P^*_S\langle M/N\rangle$
 is contained in the image, by
 $(t,y;\xi:\eta:\zeta)\mapsto (t^k,y,t^n\varphi(t,y);\xi:\eta:\zeta)$, of 
 the set defined by the equations
\[
k\xi+t^n(n\varphi+t\partial_t\varphi)\zeta=0,
\qquad
\eta+t^n\partial_y\varphi\zeta=0.
\]
The proof in the second case is similar.
\end{proof}

Given $\sigma\in S$, let $m_\sigma(S)$ denote the multiplicity of $S$ at $\sigma$.

\begin{lemma}\label{TRIVIALTRANSVERSAL}
Assume $N=\{x=0\}$, 
$S$ admits a fractional power expansion $z=x^\lambda y^\mu\varphi(x^{1/d},y^{1/d})$, 
with $\varphi(0,0)\not=0$, $(\lambda,\mu)\in\mathbb Q^2\setminus \mathbb Z^2$, $\mu\not=0$ and
Sing$^N(S)=\{y=z=0\}$.
The following statements are equivalent:
\begin{enumerate}[$(a)$]
\item $\Sigma^N_o(S)=\{(0:0:1)\}$,
\item $\Sigma^N_o(S)\not=\Sigma^N_o$,
\item $\mu\ge 1$,
\item there is a neighbourhood $U$ of $o$ such that $m_\sigma(S)=m_o(S)$ for each $\sigma\in$Sing$^N(S)\cap U$.
\end{enumerate}
\end{lemma}

\begin{proof}
The equivalence between $(c)$ and $(d)$ follows from well known facts on quasi-ordinary surfaces.
The prove that $(a)$ follows from $(c)$ is similar to the proof of Lemma \ref{TRIVIAL}. Moreover, $(a)$ implies $(b)$.

Assume $\mu<1$.
There are positive integers $m,n,k$ such that $n<k$ and
\begin{equation}\label{PARAS}
x=t^k,
\qquad
y=s^k,
\qquad
z=t^ms^n\varphi(t,s)
\end{equation}
is a parametrization of $S$. Replacing (\ref{PARAS}) in (\ref{XIDXX})
we conclude that
\begin{equation}\label{EQ}
k\eta+t^ms^{n-k}(m\varphi+s\partial_s\varphi)\zeta=0.
\end{equation}
Setting $s=Ct^{m/(k-n)}$ in (\ref{EQ}) and taking limits we conclude that
\[
k\eta+C^{n-k}m\varphi(0,0)\zeta=0.
\]
Hence the limits of tangents along the curves 
\[
x=t^k,
\qquad
y=Ct^{mk/(k-n)},
\qquad
z=t^{mk/(k-n)}\varphi(t,Ct^{m/(k-n)})
\]
of $S$ are dense in $\Sigma_o^N$.
Therefore $(b)$ is false.
\end{proof}

\begin{theorem}\label{CHANGESMOOTH}
Assume $o\in S\cap N$ is a smooth point of $N$ and the singular locus of $S$ is contained in $N$ in a neighbourhood of $o$.
Let $\rho$ be a projection compatible with $N$.
Then $\Sigma_o^N(S)=\Sigma_o^N$
if and only if $o\in \Xi_\rho^N(S)$.
\end{theorem}
\begin{proof}
We can assume $S$ irreducible.
There is a system of local coordinates $(x,y,z)$ centered at $o$
 such that $N=\{ x=0\}$ and
$S=\{F=0\}$, for some $F\in\mathbb C\{x,y,z\}$.

Assume $o\not\in \Xi_\rho^N(S)$. Then $S$ admits a fractional power series expansion $z=\varphi(x^{1/d},y)$, where
$\varphi\in \mathbb C\{ x^{1/d},y\}$, for some positive integer $d$, and there are $a_\ell\in\mathbb C\{y\}$ such that $x=t^k$,
\begin{equation}\label{AK}
\varphi (t,y)=\sum_\ell a_\ell(y)t^\ell \qquad \hbox{and} \qquad
a_\ell=0 \hbox{ or } a_\ell(0)\not=0,
\end{equation}
for each $\ell$. By Lemma \ref{TRIVIAL}, $\Sigma^N_o(S)=\langle dz \rangle$.

Set $\rho_\lambda (x,y,z)=(x,y+\lambda z )$. 
In order to prove the other implication it is enough to show that, for each $\lambda\in\mathbb C$,
\begin{equation}\label{MAGICEQUIV}
o\not\in\Xi^N_\rho(S) ~ \hbox{ if and only if } ~ o\not\in\Xi^N_{\rho_\lambda}(S).
\end{equation}
Notice that if $o\in\Xi_{\rho_\lambda}(S)$ for each $\lambda\in\mathbb C$, 
\[
\langle dy-\lambda dz\rangle \in \Sigma_o(S), \qquad \hbox{for each }\lambda \in\mathbb C.
\]
Let us prove (\ref{MAGICEQUIV}).
Assume $o\not\in \Xi_\rho^N(S)$.
Set $\Phi(x,y,z)=(x,y-\lambda z,z)$, $G=F\circ \Phi$.
Set $h(t,y,z)=z-\varphi (t,y+\lambda z)$.
There is $\psi\in\mathbb C\{t,y\}$ such that $h(t,y+\lambda\psi(t,y),z)=0$.
Hence
\begin{equation}
\psi(t,y)=\varphi(t,y+\lambda\psi(t,y))
\end{equation}
and $G$ admits the fractional power series expansion $z=\psi(x^{1/d},y)$. 
Moreover,
there are $b_k\in\mathbb C\{y\}$ such that
\[
\psi (t,y)=\sum_k b_k(y)t^k \qquad \hbox{and} \qquad
b_k=0 \hbox{ or } b_k(0)\not=0.
\]
Therefore
$\Delta_zG\subset\{x=0\}$.
\end{proof}

\begin{theorem}\label{CHANGETRANS}
Assume Sing$^N(S)$ is transversal to $N$ at $o$.
Let $\rho$ be a projection compatible with $N$.
Then $\Sigma_o^N(S)\not=\Sigma_o^N$ if and only if $o\not\in \overline{\Xi^N_\rho(S)\setminus \textrm{Sing}^N(S)}$ and there is an open neighbourhood $U$ of $o$ such that $m_{\sigma}(S)= m_o(S)$, for each $\sigma\in$Sing$^N(S)\cap U$.
\end{theorem}

\begin{proof}
We can assume that $S$ irreducible,
 $N=\{ x=0\}$ and $\textrm{Sing}^N(S)=\{y=z=0\}$.
 Therefore $S$ admits a fractional power expansion 
\begin{equation}\label{PUISEUXXY}
z=x^{\mu}y^{\nu}\varphi(x^{1/d},y^{1/d}),
\end{equation}
where $\varphi\in\mathbb C\{x,y\}$, $\varphi(0,0)\neq 0$ and $\mu,\nu\in\mathbb Q$,
for some positive integer $d$. 
After an eventual a change of coordinates, we can assume $(\mu,\nu)\not\in \mathbb Z^2$.

Let us show that the condition is sufficient.
Since $\textrm{Sing}^N(S)$ is transversal to $N$ at $o$, $\nu\neq 0$. 
Since  $\sigma\in \textrm{Sing}^N(S)\cap U$ implies $m_{\sigma}(S)=m_o(S)$, $\nu\geq 1$. 
Hence, by Lemma \ref{TRIVIALTRANSVERSAL}, $\Sigma^N_o(S)=\langle dz \rangle$. 

\noindent
Assume $\Sigma_o^N(S)\not=\Sigma_o^N$. 
By Lemma \ref{TRIVIALTRANSVERSAL}, $\nu\geq 1$.
Hence the condition on $m_\sigma(S)$ is verified.
Moreover, there is $\lambda\in\mathbb C$ such that
\[
o\not\in \overline{\Xi^N_{\rho_\lambda}(S)\setminus \textrm{Sing}^N(S)}.
\]
Therefore it is enough to show that, for each $\lambda\in\mathbb C$,
\begin{equation}\label{MAGICEQUIVXY}
o\not\in \overline{\Xi^N_\rho(S)\setminus \textrm{Sing}^N(S)} ~ \hbox{ if and only if } ~
o\not\in \overline{\Xi^N_{\rho_\lambda}(S)\setminus \textrm{Sing}^N(S)}.
\end{equation}
Set $x_*=x, ~ y_*=y-\lambda z, ~ z_*=z, ~ x=t^d, ~ y=s^d, ~   y_*=s_*^d$,
\[
f(t,s)=t^{\mu d}s^{(\nu-1)d}\varphi(t,s).
\]
Assume that there is $\psi\in\mathbb C\{t,s_*\}$ such that
\begin{equation}\label{CHANGEPARXY1}
s=s_*(1+ t^{\mu d}s_*^{(\nu-1)d}\varphi(t,s_*)\psi(t,s_*)).
\end{equation}
There are $a_k\in\mathbb C^*$, $b_{\alpha,\beta}\in\mathbb C$, with $k\ge 1$, $\alpha,\beta\ge 0$ such that $b_{0,0}=0$,
\[
(1-u)^{1/d}=\sum_ka_ku^k
\qquad \hbox{and}\qquad
f(t,s)=\sum_{\alpha,\beta}b_{\alpha,\beta}t^\alpha s^\beta.
\]

Replacing (\ref{PUISEUXXY}) in $y_*=y-\lambda z$, we conclude that

\begin{equation}\label{SSTARS}
s_*=s(1+\sum_ka_k\lambda^kf^k(t,s)),
\end{equation}

Replacing (\ref{CHANGEPARXY1}) in (\ref{SSTARS}), we show that

\begin{equation}\label{CHANGEPARXY2}
\psi f+(1+\psi f) (1 + \sum_{k\geq 1} a_k \lambda^k  f^k(t,s_*(1+\psi f)))=0.
\end{equation}

There are $c_{\alpha,\beta}\in\mathbb C$, depending only on the $b_{\alpha,\beta}$'s, such that
\begin{equation}\label{CHANGEPARXY3}
f(t,s_*(1+\psi f(t,s_*))=f(t,s_*)[1+\sum_{\alpha,\beta\geq 0}\sum_{j=1}^{\beta}c_{\alpha\beta}t^{\alpha}s^{\beta}_*\psi^{j}f^{j-1}(t,s_*)].
\end{equation}

Setting
\begin{displaymath}
\varepsilon_{\psi}(t,s_*)=1+\sum_{\alpha,\beta\geq 0}\sum_{j=1}^{\beta}c_{\alpha\beta}t^{\alpha}s^{\beta}_*\psi^{j}f^{j-1}(t,s_*).
\end{displaymath}
we can rewrite equality (\ref{CHANGEPARXY2}) as 
\begin{equation}\label{CHANGEPARXY4}
\psi +(1+\psi f)\sum_{k\geq 1} a_k \lambda^k f^{k-1}\varepsilon_{\psi}^{k}=0,
\end{equation}

\noindent
Hence $\psi$ is well defined. Furthermore, $\psi(0,0)=\lambda / d$. Hence, for each $\lambda \in\mathbb C^{\ast}$, $S$ admits a fractional power series expansion $z_*= x_*^{\mu}y_*^{\nu}\phi(x_*^{1/d},y_*^{1/d})$ with $\phi(0,0)\neq 0$. Therefore 
$o\not\in \overline{\Xi^N_{\rho_\lambda}(S)\setminus \textrm{Sing}^N(S)}$.
\end{proof}

\section{A logarithmic version of a classical result}\label{LOGARITHMIC}

Let $M$ be a germ of a complex manifold of dimension $3$ at a point $o$.
Let $\pi:\widetilde M\to M$ be the blow up of $M$ at $o$.
Set $E=\pi^{-1}(o)$.
Let $\widehat\pi: \mathbb P^*\langle \widetilde M/E\rangle \to \mathbb P^*M$
be the blow up of $\mathbb P^*M$ along the Legendrian variety $\pi_M^{-1}(o)$.
If follows from Proposition 9.4 of \cite{ON1}  that 
diagram (\ref{FIRSTBU}) commutes.
We will also denote by $\widehat\pi$ the bimeromorphic map 
from a dense open set of 
$T^*\langle \widetilde M/E\rangle$ into $T^*M$ that induces $\widehat\pi$.

Let $\ell$ be a line of $T_oM$ that contains the origin.
Let $\Sigma_\ell$ be the set of planes of $T_oM$ that contain $\ell$.
Let $\gamma$ be the germ at $o$ of a smooth curve of $M$ with tangent space $\ell$.
The point $o_\ell$ where the strict transform of $\gamma$ intersects $E$ does not depend on $\gamma$.

\begin{theorem}\label{CLASSICAL}
Let $S$ be a surface of $M$. Then $\Sigma_o(S)$ is the union of the dual of the projectivization of $C_o(S)$ and a finite set of projective lines of $\mathbb P^*_oM$.

Moreover, $\Sigma_\ell\subset \Sigma_o(S)$ if and only if $\Sigma^E_{o_\ell}(\widetilde S)=\Sigma^E_{o_\ell}$.
\end{theorem}

\begin{proof}
Let $\ell \in \mathbb P(T_oM)$.
Choose local coordinates $(x,y,z)$ of $M$ such that $\ell=\{y=z=0\}$.
Setting $x_0=x,y_0=y/x,z_0=z/x$, $(x_0,y_0,z_0)$ is a system of local coordinates on an affine set $U$ of $\widetilde M$, centered at $o_\ell$ such that $E\cap U=\{x_0=0\}$. Let
\[
\xi_0\frac{dx_0}{x_0} ~ + ~ \eta_0 dy_0 ~ + ~ \zeta_0 dz_0
\]
be the canonical $1$-form of $T^*\langle \widetilde M/E \rangle$
in a neighbourhood of $o_\ell$. Since 
\[
\widehat\pi^*\left(  
\xi dx+\eta dy + \zeta dz
\right)
~ = ~
(x\xi+y\eta+z\zeta)\frac{dx_0}{x_0}+x\eta dy_0+x\zeta dz_0,
\]
in a neighbourhood of $o_\ell$, $\widehat{\pi}$ induces a map $\widehat{\pi}_\ell:\mathbb P_{o_\ell}^*\langle \widetilde M/E\rangle \to \mathbb P_o^*M$ given by 
\[
\widehat\pi_\ell (\xi_0:\eta_0:\zeta_0)=(\xi_0:\eta_0:\zeta_0).
\]
Therefore 
\begin{equation}\label{FITS}
\widehat\pi(\mathbb P^*_{o_\ell}E)=\Sigma_\ell.
\end{equation}
Since 
$
\widehat \pi
\left(
\mathbb P^*_{\widetilde S}\langle \widetilde M/ E \rangle
\right)
 ~ = ~ \mathbb P^*_SM
$,
\[
\widehat \pi
\left(
\mathbb P^*_{\widetilde S}\langle \widetilde M/ E \rangle
\cap \mathbb P^* E\right)
 ~ = ~ \Sigma_o(S).
\]
Since
$\mathbb P^*_{\widetilde S} \langle \widetilde M/ E \rangle
\cap \pi^{-1}_{\widetilde M} (E) = \mathbb P^*_{\widetilde S}\langle M/ E \rangle
\cap \mathbb P^* E$ is a Legendrian variety of $\mathbb P^* E$ and  
\[
\pi_{\widetilde M}
\left(
\mathbb P_{\widetilde S}\langle \widetilde M/ E \rangle
\cap \mathbb P^* E\right)
 ~ = ~
 \widetilde S\cap E,
\]
 there are $\ell_1,...,\ell_k\in \mathbb P\left( T_oM \right)
 $ 
 such that
\[
\mathbb P^*_{\widetilde S}\langle \widetilde M/ E \rangle \cap \mathbb P^* E
~ = ~
\mathbb P^* _{\widetilde S\cap E}E
~ \cup ~ \cup_{i=1}^k\mathbb P^*_{o_{\ell_i}}E.
\]
Since $\widetilde S \cap E \simeq$ Proj$(C_o(S))$, 
$\widehat \pi
\left(
\mathbb P^*( \widetilde S\cap E) \right)
$
equals the dual of Proj$(C_o(S))$. The Theorem follows from (\ref{FITS}).
\end{proof}

\section{The First Sequence of Blow-Ups}\label{FSB}

Let $M$ be a germ of a complex manifold of dimension $3$ at a point $o$.
Let $\rho$ be a submersion of $M$ into a smooth surface $X$.
We will fix systems of local coordinates $(x,y,z)$, $[(x,y)]$ on $M$ $[X]$ such that
 $o=(0,0,0)$ and $\rho(x,y,z)=(x,y)$. Let
\begin{equation}\label{THETAXY}
 \xi\frac{dx}{x}+\eta \frac{dy}{y}+ \zeta dz
\end{equation}
be the canonical $1$-form $\theta_N$ of $T^*\langle M/N\rangle$.
 Given $(a:b)\in \mathbb P^1$,
 set
 \[
\Sigma^N_o(a:b)=\{ (\xi : \eta : \zeta ) \in \mathbb{P}^*_o\langle M/N \rangle ~ : ~  b\xi + a\eta=0  \}.
 \]
Let $\pi:\widetilde M \to M$ be the blow up of $M$ along the singular locus $N^\sigma$ of $N$. 
Set $\widetilde N=\pi^{-1} (N)$. 
Let $\widehat{\pi} : T^*\langle \widetilde  M/ \widetilde   N \rangle \to  T^* \langle M/N \rangle $
 be the blow up of $ T^* \langle M/N \rangle $  along $\pi_M^{-1}(N^\sigma )$.
There is a commutative diagram

 \begin{equation}\label{BUT}
\begin{array}{ccccc}
  T^*\langle M/N\rangle  & \leftarrow &  T^*\langle \widetilde
M/\widetilde N \rangle  \\
  \downarrow &  &   \downarrow \\
  M  & \leftarrow &   \widetilde M
\end{array}
\end{equation}

Set $x_0=x$, $y_0=y/x$ and $z_0=z$. 

\begin{lemma}\label{BLOWLIMIT}
 Let $o_1\in \pi^{-1}(o)$. 
 The following statements hold:

\begin{enumerate}[	$~~~~~(a)$]
\item
The map $\widehat \pi$ induces a linear isomorphism $\widehat \pi ~ : ~ T^*_{o_1}\langle \widetilde M / \widetilde N \rangle ~ \to ~ T^*_o \langle M/N \rangle$ given by
\begin{equation}\label{LINEARISO}
( \xi_0,\eta_0,\zeta_0) \mapsto (\xi_0-\eta_0,\eta_0,\zeta_0).
\end{equation}

\item 
If $o_1=(0,0,0)$,
\begin{equation}\label{LINEARSUM}
\widehat{\pi }
\left( \Sigma^{\widetilde{N}}_{o_1} (a:b)\right) ~ =  ~\Sigma^N_o(a+b:b).
\end{equation}

 \item 

If $o_1=(0,\mu,0)$, with $\mu\not=0$,
\begin{equation}\label{LINEARP1}
\widehat{\pi}
\left( \Sigma^{\widetilde N}_{o_1}\right)
=\Sigma^N_o(1:1).
\end{equation}

\item

Assume $(a:b)\not=(1:1)$, $o_1=(0,0,0)$ and $o_2$ is the other point of $\pi^{-1}(o)$ such that $\widetilde N$ is singular at $o_2$.
Then
$
\Sigma_o^N(S)\supset \Sigma_o^N(a:b)
$
if and only if
\[
\Sigma^{\widetilde N}_{o_1}(\widetilde S) \supset \Sigma^{\widetilde N}_{o_1}(a-b,b)
\qquad \hbox{or} \qquad
\Sigma^{\widetilde N}_{o_2}(\widetilde S) \supset \Sigma^{\widetilde N}_{o_2}(a,b-a).
\]

\end{enumerate}

\end{lemma}

\begin{proof}
Statement $(a)$ follows from the fact that
\[
\widehat{\pi}^* \theta_N=(\xi+\eta)\frac{dx_0}{x_0}+\eta\frac{dy_0}{y_0}+\zeta  dz_0.
\]
in a neighbourhood of $o_1$.
Statements (b) and (c) follow from statement (a).
Since $\mathbb P^*_S\langle M/ N \rangle \subset 
\widehat{\pi}\left( \mathbb P^*_{\widetilde S}\langle \widetilde M/\widetilde N \rangle \right)$,
\[
\mathbb P^*_S\langle M/ N \rangle 
\cap \pi_M^{-1}(o)
\subset 
\widehat{\pi}\left( 
\mathbb P^*_{\widetilde S}\langle \widetilde M/\widetilde N \rangle 
\cap \pi_{\widetilde M}^{-1}(\pi^{-1}(o))
\right).  
\]
Therefore
\begin{equation}\label{CONTAINED}
\Sigma^N_o(S) \subset  \cup_{s\in \widetilde S\cap \pi^{-1}(o)}\widehat{\pi}\left( \Sigma^{\widetilde N}_s(\widetilde S)\right).
\end{equation}
Statement (d) follows from (\ref{CONTAINED}) and statements (b) and (c).
\end{proof}

\begin{blank}\label{PRIMEIRA}\em
Set $M_0=M,~N_0=N,~X_0=X,~S_0=S,~\rho_0=\rho$.

Let $\rho_\ell:M_\ell\to X_\ell$ be a holomorphic submersion,
Let $N_\ell$ be a normal crossings divisor of $M_\ell$.
Let $S_\ell$ be a surface of $M_\ell$.

Let $D_\ell$ be intersection of the closure of $\Delta_{\rho_\ell}(S_\ell)\setminus \rho(N_\ell)$ and $\rho(N_\ell^\sigma)$.

Let $\tau_{\ell+1}:X_{\ell+1}\to X_\ell$ be the blow up of $X_\ell$ 
along  $D_\ell$.
Let $\pi_{\ell+1}:M_{\ell+1}\to M_\ell$ be the blow up of $M_\ell$ along $\rho_\ell^{-1}(D_\ell)$.
Let $S_{\ell+1}$ be the strict transform of $S_\ell$ by $\pi_{\ell+1}$.
By the universal property of the blow-up there is a map $\rho_{\ell+1}:M_{\ell+1}\to X_{\ell+1}$
such that $\rho_\ell\pi_{\ell+1}=\tau_{\ell+1} \rho_{\ell+1}$.
Moreover, $\rho_{\ell+1}$ is a submersion.
Hence we can iterate the process.
\end{blank}

There is an integer $L$ such that $D_L=\emptyset$. 
Hence the procedure described in paragraph \ref{PRIMEIRA} will terminate.
Set $\pi=\pi_1\circ\cdots\circ\pi_L$. 
We call the map $\pi:M_L\to M$ the \em first sequence of blow-ups. \em

\begin{theorem}\label{SIGMAAB}
Assume $N$ has two irreducible components and $S\cap N^{\sigma}=\{o\}$.
 There are positive integers $a_1,...,a_n,b_1,...,b_n$ such that
\[
\Sigma^N_o (S) \subset \cup^n_{i=1}\Sigma^N_o(a_i:b_i).
\]
Moreover, $\Sigma^N_o (S)$ contains a projective line if and only if there is a regular point $o_1$ of $N_L$ such that
$\Sigma^{N_L}_{o_1} (S_L)=\Sigma^{N_L}_{o_1}$.
\end{theorem}
\begin{proof}
Let $\sigma\in S_{L}$. 
If $\sigma$ is a singular point of $N_L$,
\[
\Delta_{\rho_L}(S_L)=\rho_L(N_L).
\]
Assuming $N_L=\{xy=0\}$ in a neighbourhood of $\sigma$, 
it follows from Lemma \ref{TRIVIAL} that $\Sigma^{N_L}_{\sigma}=\{(0:0:1\}$.\\
Assume $L=1$. It follows from statements $(a)$ and $(c)$ of Lemma \ref{BLOWLIMIT} that 
$\Sigma_o^N(S)\subset \Sigma_o^N(1:1)$.

The induction step follows from statement $(d)$ of Lemma \ref{BLOWLIMIT}.
\end{proof}

\section{The Non Degenerated Case}\label{NDC}

Let $N$ be a smooth divisor of a manifold $M$ of dimension $3$.
Let $S$ be a surface of $M$ that does not contain $N$.
Let $o\in S\cap N$.
Let $\pi^0:M_0\to M$ be the blow up of $M$ with center $o$.
Let $\widetilde N$ be the strict transform of $N$.
Let $S_0$ denote the strict transform of $S$.
Set $N_0=(\pi^0)^{-1}(N)$, $E=(\pi^0)^{-1}(o)$.

We say that $S$ is \em non degenerated \em at $o$ if $C_{o}(S)$ does not contain $C_o(N)$.

\begin{lemma}\label{PCIRC}
There is an open set $\mathbb P^\circ \langle M_0/N_0\rangle$ of 
$\mathbb P^* \langle M_0/N_0\rangle$ 
and an holomorphic map 
$\pi_0:\mathbb P^\circ \langle M_0/N_0\rangle \to \mathbb P^* \langle M/N\rangle$
such that the diagram 
\begin{equation}\label{LOGBU}  
\begin{array}{ccccc}
  \mathbb{P}^*\langle M/N\rangle  &  \overset{\pi_0}{\longleftarrow} &  \mathbb{P}^\circ\langle 
M_0 / N_0 \rangle  \\
 \pi_M \downarrow \qquad &  &   \qquad \downarrow \pi_{M_0} \\
  M  & \overset{\pi^0}{\longleftarrow} &   M_0
\end{array}
\end{equation}
commutes and
\begin{enumerate}[$(a)$]
\item
$\mathbb P^* (M\setminus N)\hookrightarrow \mathbb P^\circ \langle M_0/N_0\rangle$,
\item
$\pi_0|_{\mathbb P^*(M_0\setminus E)}: 
\mathbb P^*(M_0\setminus E) \to {\mathbb P^*(M\setminus \{o\}})$ 
is a contact transformation,
\item
for each Legendrian surface $\Gamma$ of $\mathbb P^* \langle M/N\rangle$, $\Gamma\subset \pi_0(\mathbb P^\circ \langle M_0/N_0\rangle)$,
\item
for each Legendrian surface $\Gamma_0$ of $\mathbb P^* \langle M_0/N_0\rangle$, 
$\Gamma_0\subset \mathbb P^\circ \langle M_0/N_0\rangle$.
\end{enumerate}
\end{lemma}
\begin{proof}
Assume $M$ is an affine set with coordinates $(x,y,z)$ and  $N=\{x=0\}$.  
The manifold $M_0$ is the gluing of the open affine sets $V_i$, $i=1,2,3$,  with coordinates $(x_i,y_i,z_i)$ such that
\begin{enumerate}
\item
$x_1=x,y_1=y/x,z_1=z/x$;
\item
$x_2=x/y,y_2=y,z_2=z/y$;
\item
$x_3=x/z,y_3=z,z_3=y/z.$
\end{enumerate}
Let $\pi'_0:\mathbb P^*(M_0\setminus E) \to {\mathbb P^*(M\setminus \{o\}})$  be the contact transformation such that
\[
\pi_M \circ \pi'_0 = \pi^0 \circ \pi_{M_0}.
\]
Let $\pi_{0,i}$ be the restriction of $\pi'_0$ to $\pi_{M_0}^{-1}(V_i\setminus N_0)$, $i=1,2,3$.
 Since
\begin{displaymath}
{\pi}_{0,1}^{\ast}(\xi\frac{d x}{x}  + \eta d y  +  \zeta d z)=
(\xi+y\eta+z\zeta)\frac{d x_1}{x_1}  + x \eta d y_1  +  x \zeta d z_1,
\end{displaymath}
$\xi_1=\xi+y\eta+z\zeta$, $\eta_1=x \eta$ and $\zeta_1=x \zeta$. Hence
\[
{\pi}_{0,1}(x_1,y_1,z_1;\xi_1:\eta_1:\zeta_1)=
(x_1,x_1y_1,x_1z_1;x_1(\xi_1- y_1 \eta_1- z_1 \zeta_1):\eta_1:\zeta_1).
\]
Therefore ${\pi}_{0,1}$ is defined outside of $B_1=\{x_1=\eta_1=\zeta_1=0\}$.

The canonical $1$-form of 
 $T^*\langle V_2/N_0\cap V_2\rangle)$
 $[T^*\langle V_2\setminus \widetilde N/(E\setminus \widetilde N)\cap V_2 \rangle]$ 
 is 
 \[
 \xi_2\frac{dx_2}{x_2}+\eta_2\frac{dy_2}{y_2}+\zeta_2dz_2
 \qquad
 \left[ \xi'_2  dx_2+\eta_2 \frac{dy_2}{y_2} +\zeta_2dz_2, 
 \right],
 \]
 where $\xi_2=x_2\xi'_2$ if $x_2\not=0$.
 Since
 \[
 {\pi}_{0,2}^{\ast}(\xi\frac{d x}{x}  + \eta d y  +  \zeta d z)=
\xi\frac{d x_2}{x_2}  +(\xi+y\eta+z\zeta) \frac{d y_2}{y_2}  +  y\zeta d z_2,
\]
${\pi}_{0,2}$ is given by $x=x_2y_2$, $y=y_2$, $z=y_2z_2$,
\begin{equation}\label{PIZERO}
\xi=y_2\xi_2,
\qquad
\eta=\eta_2-\xi_2-z_2\zeta_2,
\qquad
\zeta=\zeta_2.
\end{equation}
Therefore ${\pi}_{0,2}$ is defined outside of $\{y_2=\eta_2-\xi_2=\zeta_2=0\}$.
Since $\xi_2=x_2\xi'_2$, when $x_2\neq 0$, ${\pi}_{0,2}$ is defined outside of the union of the sets
\[
B_2=\{x_2=y_2=\eta_2-\xi_2=\zeta_2=0\},
\]
\[
B'_2=\{y_2=\eta_2-x_2\xi'_2=\zeta_2=0,x_2\neq 0\}.
\]
Let $\Gamma_0$ be a Legendrian variety of $\mathbb P^*\langle M_0/N_0\rangle$.
Since $B_1\cap \{\xi_1=0\}=\emptyset$, $\Gamma_0\cap B_1=\emptyset$ by Lemma \ref{RESIDUAL}.
By a similar argument $\Gamma_0\cap B'_2=\emptyset$.
By Theorem \ref{SIGMAAB}, $\Gamma_0\cap B_2=\emptyset$. We apply the same reasoning to $V_3$.
\end{proof}

\begin{lemma}\label{NONTRIVCONE}
 If $\Sigma^{N}_{o}(S)$ is finite, $C_{o}(S)$ is a union of planes.
\end{lemma}

\begin{proof}
Let $L$ be an irreducible component of $S_0\cap E$. Notice that if $L= \widetilde N \cap E$, 
$L$ is the projectivization of $C_{o}(N)$. Assume $L\neq \widetilde N \cap E$.
By Bezout's Theorem, there is $\sigma\in L\cap \widetilde N$.
We can assume that $\sigma$ is the origin of $V_2$.
Let $\gamma$ be an irreducible component of the germ of $L$ at $\sigma$.
There is a local parametrization of $\gamma$  of the type
\begin{displaymath}\label{PARAMCURVE0}
x_2=\varepsilon_1(t) t^{k_1},\; y_2=0,\;z_2=\varepsilon_2(t) t^{k_2},
\end{displaymath}
such that $k_1,k_2$ are positive integers, $(k_1,k_2)=1$, $\varepsilon_1,\varepsilon_2\in\mathbb C\{t\}$ and $\varepsilon_1\not \equiv 0$. Furthermore, we can assume
\begin{enumerate}
\item[($a$)]
if $\varepsilon_2\equiv 0$, $\varepsilon_1\equiv 1$ and $k_1=1$,
\item[($b$)]
if $k_j\geq k_n$, $\varepsilon_n\equiv 1$ and $\varepsilon_j(0)\neq 0$,
\end{enumerate}
where $j,n\in\{1,2\}$ and $j\neq n$. 
Therefore $\mathbb P^*_{\gamma}( N_0\setminus \widetilde N)$ admits a parametrization
\begin{displaymath}
x_2=\varepsilon_1(t) t^{k_1},\;y_2=0,\;z_2=\varepsilon_2(t) t^{k_2},\;\xi_2=-\delta(t)t^{k_2}\zeta_2, \;\eta_2=0,
\end{displaymath}
where $\delta(t)=\varepsilon_1(t)(k_2\varepsilon_2(t)+\varepsilon'_2(t) t)/(k_1\varepsilon_1(t)+\varepsilon'_1 (t)t)$. 

Since $\Sigma_o^N(S)\neq\Sigma_o^N$, $\pi_0(\mathbb P^*_{\gamma}( N_0\setminus \widetilde N))$ is a point.
By (\ref{PIZERO}), $\delta=\varepsilon_2$. Therefore
\begin{equation}\label{ZEROEQU0}
(k_2-k_1)\varepsilon_1\varepsilon_2+t(\varepsilon_1\varepsilon'_2-\varepsilon'_1\varepsilon_2)= 0.
\end{equation}
If $\varepsilon_2=0$, $\gamma$ is contained in a projective line. Hence $L$ is a projective line. 
Assume $\varepsilon_2(0)\neq 0$. 
Then $k_2-k_1=0$. 
Hence we can assume $k_2=1$, $\varepsilon_1= 1$. 
Replacing $k_1$, $k_2$ and $\varepsilon_1$ in (\ref{ZEROEQU0}), we conclude that $\varepsilon'_2=0$.
Therefore $\varepsilon_2\in\mathbb C^{\ast}$. 
Hence $L$ is a projective line. 
\end{proof}

\begin{theorem}\label{TRIVCONE}
 Assume  $S$ is \em non degenerated \em at $o$.
  Then  $\Sigma^{N}_{o}(S)$ is finite if and only if $C_o(S)$ is a union of planes and for each $\sigma\in S_0\cap E$, 
  $\Sigma^{N_0}_{\sigma}(S_0)$ is finite.
\end{theorem}

\begin{proof}
By Lemma \ref{NONTRIVCONE}, $S_0\cap E$ is a union of projective lines $L_i$, $i=1,...,n$.
Since $S$ is non degenerated, there are $\sigma_1,...,\sigma_n$ such that
\[
S_0\cap E\cap \widetilde N=\{\sigma_1,...,\sigma_n\}.
\]
Moreover,
there are points $o_{1},...,o_m$ of $(S_0\cap E)\setminus \widetilde N$ such that
\[
\mathbb P^*_{S_0}\langle M_0/N_0 \rangle \cap \pi_{M_0}^{-1}(E)=
\cup_i\mathbb P^*_{L_i}E\cup\cup_i\Sigma_{\sigma_i}^{N_0}(S_0)\cup\cup_i\Sigma_{o_i}^{N_0}.
\]
By Lemma \ref{PCIRC},
\[
\Sigma^N_o(S) \subset \pi_0(\cup_i\mathbb P^*_{L_i}E\cup\cup_i\Sigma_{\sigma_i}^{N_0}(S_0)\cup\cup_i\Sigma_{o_i}^{N_0}).
\]
By the arguments of Lemma \ref{NONTRIVCONE}, $\pi_0(\mathbb P^*_{L_i}E)$ is a point, for $i=1,...,n$.
The type of arguments  used in Theorem \ref{CLASSICAL} show that 
$\pi_0(\Sigma_{\sigma_i}^{N_0}(S_0))$ is finite if and only if 
$\Sigma_{\sigma_i}^{N_0}(S_0)$ is finite, $i=1,...,n$ and $\pi_0(\Sigma_{o_i}^{N_0})$ is infinite for $i=1,...,m$.

Hence $\Sigma_o^N(S)$ is finite if and only if $\Sigma_{\sigma_i}^{N_0}(S_0)$ is finite, $i=1,...,n$ and $m=0$.
\end{proof}

\section{The Second  sequence of Blow-ups}\label{SSB}

We will introduce a generalization for surfaces of a construction for curves that was introduced in \cite{CNP}.
Given non negative integers $a_0,a_1,...,a_g$ assume that $a_1,...,a_g\ge1$ or $g=1$, $a_1=0$.
Set $[a_0,0]=\infty$, $[a_0]=a_0$. If $g,a_g\ge 1$, set 
\[
[a_0,...,a_g]=a_0+[a_1,...,a_g]^{-1}.
\]
Assuming $a_g\ge 2$, $[a_0,...,a_g]=[a_0,...,a_g-1,1]$. 
It is well known that each positive rational number is described by exactly two continuous fractions.

If $\alpha=[a_0,...,a_g]$ we call \em length \em of $\alpha$ to $|\alpha|=a_0+\cdots+a_g$.
Set $n_\infty=1$, $d_\infty=0$.
If  $\alpha=a/b$ where $a,b$ are positive integers such that $(a,b)=1$, set $n_\alpha=a$, $d_\alpha=b$, $e_\alpha=a+b$.

Let $\alpha$ be a positive rational number, $\alpha=[a_0,...,a_g]$. If $\alpha$ is an integer, set $\alpha_\omega=\alpha-1$, $\alpha_\pi=\infty$. Otherwise, $g\ge 1$. Moreover, we can assume $a_g\ge 2$. 
Set $\alpha_\omega=[a_0,...,a_g-1]$ and $\alpha_\pi=[a_0,...,a_{g-1}]$ if $g$ even, otherwise set $\alpha_\omega=[a_0,...,a_{g-1}]$ and $\alpha_\pi=[a_0,...,a_g-1]$.

Assume $\alpha=1$ or $a_g\ge 2$. Set $\alpha_s=[a_0,...,a_{g-1},a_g-1,2]$ and $\alpha_b=[a_0,...,a_{g-1},a_g+1]$ if $g$ even, otherwise set $\alpha_s=[a_0,...,a_{g-1},a_g+1]$ and $\alpha_b=[a_0,...,a_{g-1},a_g-1,2]$.
Notice that $\alpha_s$ and $\alpha_b$ are the only rationals such that ${\alpha_s}_{\pi}={\alpha_b}_{\omega}=\alpha$. Moreover, ${\alpha_s}_{\omega}=\alpha_{\omega}$, ${\alpha_b}_{\pi}=\alpha_{\pi}$,
\[
\alpha_s=\frac{n_\alpha+n_{\alpha_\omega}}{d_\alpha+d_{\alpha_\omega}},\; \alpha_b=\frac{n_\alpha+n_{\alpha_\pi}}{d_\alpha+d_{\alpha_\pi}},
\]
$e_{\alpha}+e_{{\alpha}_{\pi}}=e_{{\alpha}_{b}}$ and $e_{\alpha}+e_{{\alpha}_{\omega}}=e_{{\alpha}_{s}}$.

\begin{blank}\label{SECONDPROCESS}\em
Let $M$ be the germ of a complex analytic manifold of dimension $3$ at a point $o$.
Let $N$ be a smooth surface of $M$.
Let $S$ be a singular surface of $M$.
Assume that $N$ is not an irreducible component of $S$ 
and $C_o(S)\supset C_o(N)$.

Let $\pi^0:M^0\to M$ be the blow up of $M$ at $o$.
Set $N^0=(\pi^0)^{-1}(N)$ and $E^0=(\pi^0)^{-1}(o)$.
Let $S^0$ be the strict transform of $S$ by $\pi^0$.
Notice that $S^0$ contains the singular locus $Z^0$ of $N^0$.

Let $N^k$ be a normal crossings divisor of a manifold $M^k$ of dimension $3$.
Let $S^k$ be a singular surface of $M^k$.
Let $Z^k$ be the union of the connected components of the singular locus of $N^k$ that are contained in $S^k$.

We iterate the process defining $\pi^{k+1}:M^{k+1}\to M^k$ as the blow up of $M^k$ along $Z^k$,
defining $S^{k+1}$ as the strict transform of $S^k$ by $\pi^{k+1}$ and 
setting $N^{k+1}=(\pi^{k+1})^{-1}(N^k)$, $E^{k+1}=(\pi^{k+1})^{-1}(Z^k)$.

This process  will terminate after a finite number $k_0$ of steps.

The intersection of $S^{k_0}$ with the singular locus of $N^{k_0}$ is a finite set. 
We will now perform the first sequence of blow-ups at each point of this intersection.

We obtain in this way a map $ \pi:  \widetilde M\to M$, 
a normal crossings divisor $\widetilde N$ of $ \widetilde M$
 and a singular surface $ \widetilde S$ of $ \widetilde M$.
We call $\pi:  \widetilde M\to M$ the \em second sequence of blow-ups. \em
 \end{blank}

Let $M^{(\alpha)}$ be the gluing of the affine sets $U_{\alpha,i}$, with coordinates 
$(u_{\alpha,i},v_{\alpha,i},w_{\alpha,i})$, $i=1,2,3,4$,  by the transformations

\noindent  $v_{\alpha,3}=v_{\alpha,1}w_{\alpha,1}^{e_\alpha}$, \qquad 
       $w_{\alpha,3}=w_{\alpha,1}^{-1}$, \qquad 
       $u_{\alpha,3}=u_{\alpha,1}w_{\alpha,1}^{-e_{\alpha_\pi}} $;

\noindent $v_{\alpha,4}=v_{\alpha,2}w_{\alpha,2}^{-e_\alpha}$, \qquad 
       $w_{\alpha,4}=w_{\alpha,2}^{-1}$, \qquad 
       $u_{\alpha,4}=u_{\alpha,2}w_{\alpha,2}^{e_{\alpha_\omega}} $;

\noindent $v_{\alpha,2}=v_{\alpha,1}^{-1}, \qquad w_{\alpha,2}=w_{\alpha,1}, \qquad u_{\alpha,2}=u_{\alpha,1}v_{\alpha,1}$.

\noindent
Let $E^{(\alpha)}$, $N^\alpha_b$, $N^\alpha_s$ be defined by

\noindent  $E^{(\alpha)}\cap U_{\alpha,i}=\{u_{\alpha,i}=0\}$, $i=1,2,3,4$;

\noindent
$N^\alpha_b\cap U_{\alpha,i}=\{v_{\alpha,i}=0\} \quad i=1,3, \qquad N^\alpha_b\cap U_{\alpha,i}=\emptyset, \quad i=2,4$;

\noindent
$N^\alpha_s\cap U_{\alpha,i}=\{v_{\alpha,i}=0\} \quad i=2,4, \qquad N^\alpha_s\cap U_{\alpha,i}=\emptyset, \quad i=1,3$;

\noindent
Set 
\[
Z^{\alpha_b}=E^{(\alpha)}\cap N^\alpha_b, \qquad
Z^{\alpha_s}=E^{(\alpha)}\cap N^\alpha_b,
\]
\[
M^{(\alpha)}_b=M^{(\alpha)}\setminus N^\alpha_s, \qquad
M^{(\alpha)}_s=M^{(\alpha)}\setminus N^\alpha_b.
\]
We will denote by 
\[
N^{\alpha_b}_s \quad
[N^{\alpha_s}_s,~N^{\alpha_b}_b,~N^{\alpha_s}_b ]
\]  
the strict transform of $E^{(\alpha)}$ $[N^\alpha_s,N^\alpha_b,E^{(\alpha)}]$
by the blow up of 
$M^{(\alpha)}_b$ $[M^{(\alpha)}_s,M^{(\alpha)}_b,M^{(\alpha)}_s]$
along
$Z^{\alpha_b}$ [$Z^{\alpha_s}$
,$Z^{\alpha_b}$,$Z^{\alpha_s}$].

\begin{lemma}
Assume $M$ is an affine set with coordinates $(x,y,z)$,
$N=\{x=0\}$, $o$ is the origin and $k\ge 1$.
Then the following statements hold:

\noindent
There are finite sets $I^k_b,I^k_s\subset \mathbb Q$ such that $M^k$ is a gluing of the manifolds
\[
M^{(\alpha)}_b, 
~~ 
\alpha\in I^k_b,
\qquad
M^{(\alpha)}_s, 
~~ 
\alpha\in I^k_s;
\]

\noindent
For each $\alpha\in I^k_b$, $E^k\cap M^{(\alpha)}_b\subset E^{(\alpha)}$, $N^k\cap M^{(\alpha)}_b\subset E^{(\alpha)}\cup N^\alpha_b$.

\noindent
For each $\alpha\in I^k_s$, $E^k\cap M^{(\alpha)}_s\subset E^{(\alpha)}$, $N^k\cap M^{(\alpha)}_s\subset E^{(\alpha)}\cup N^\alpha_s$.

\noindent
For each $\alpha$,
$Z^{\alpha}$ is a projective line.

\noindent
The restriction
$\pi^{\alpha_b}:M^{(\alpha_b)}\to M^{(\alpha)}\setminus N^\alpha_s$ of $\pi^{k+1}:M^{k+1}\to M^k$ is given by

\noindent
$u_{\alpha,i}=u_{\alpha_b,i}$,
\qquad
$v_{\alpha,i}=u_{\alpha_b,i}v_{\alpha_b,i}$
\qquad
$w_{\alpha,i}=w_{\alpha_b,i}$, 
\qquad 
$i=1,3$,

\noindent
$u_{\alpha,i-1}=u_{\alpha_b,i}v_{\alpha_b,i}$,
\qquad
$v_{\alpha,i-1}=u_{\alpha_b,i}$
\qquad
$w_{\alpha,i-1}=w_{\alpha_b,i}$, 
\qquad 
$i=2,4$,

\noindent
The restriction
$\pi^{\alpha_s}:M^{(\alpha_s)}\to M^{(\alpha)}\setminus N^\alpha_b$ of $\pi^{k+1}:M^{k+1}\to M^k$ is given by

\noindent
$u_{\alpha,i+1}=u_{\alpha_s,i}v_{\alpha_s,i}$,
\qquad
$v_{\alpha,i+1}=u_{\alpha_s,i}$
\qquad
$w_{\alpha,i+1}=w_{\alpha_s,i}$, 
\qquad 
$i=1,3$,

\noindent
$u_{\alpha,i}=u_{\alpha_s,i}$,
\qquad
$v_{\alpha,i}=u_{\alpha_s,i}v_{\alpha_s,i}$,
\qquad
$w_{\alpha,i}=w_{\alpha_s,i}$, 
\qquad 
$i=2,4$.
\end{lemma}
\begin{proof}
The manifold $M^0$ is the gluing of the affine sets $V_i$ introduced at the proof of Lemma \ref{PCIRC}.
Remark that $Z^1$ is a projective line contained in $V_2\cup V_3$.
Moreover, $I^1_b=I^1_s=\{1\}$.
Setting 

\noindent
$u_{1,1}=x_2$, \qquad $v_{1,1}=y_2/x_2$, \qquad $w_{1,1}=z_2$;

\noindent
$u_{1,2}=y_2$, \qquad $v_{1,2}=x_2/y_2$, \qquad $w_{1,2}=z_2$;

\noindent
$u_{1,3}=x_3$, \qquad $v_{1,3}=y_3/x_3$, \qquad $w_{1,3}=z_3$;

\noindent
$u_{1,4}=y_3$, \qquad $v_{1,4}=x_3/y_3$, \qquad $w_{1,4}=z_3$.

\noindent
 we conclude that the lemma holds for $k=1$.
 
 Assume $|\alpha|=k$. 
 If $Z^{\alpha_b}\supset S^k$, [$Z^{\alpha_s}\supset S^k$] we withdraw $\alpha$ from $I^k_b$ [$I^k_s$] and include $\alpha_b$ into $I^{k+1}_b$ and $I^{k+1}_s$.
  Defining $u_{\alpha_b,i},v_{\alpha_b,i},w_{\alpha_b,i}$ in such a way that $\pi^{\alpha_b}$ is as proposed in this Lemma, we conclude that
 
 \noindent
$u_{\alpha_b,3}=u_{\alpha,3}=u_{\alpha,1}w_{\alpha,1}^{-e_{\alpha_\pi}}=u_{\alpha_b,1}w_{\alpha_b,1}^{-e_{{\alpha_b}_\pi}}$,

\noindent
$v_{\alpha_b,3}=v_{\alpha,3}u_{\alpha,3}^{-1}=
v_{\alpha,1}w_{\alpha,1}^{e_\alpha}u_{\alpha,1}^{-1}w_{\alpha,1}^{e_{\alpha_\pi}}=
v_{\alpha,1}u_{\alpha,1}^{-1}w_{\alpha,1}^{e_\alpha+e_{\alpha_\pi}}=v_{\alpha_b,1}w_{\alpha_b,1}^{e_{\alpha_b}}$

\noindent
and $w_{\alpha_b,3}=w_{\alpha,3}=w_{\alpha,1}=w_{\alpha_b,1}$.
\end{proof}

Let $\pi^{(\alpha)}:E^{(\alpha)}\to Z^{\alpha}$ be the restriction of $\pi^\alpha$. 
Let $C$ be an irreducible curve of $E^{(\alpha)}$. 
We say that $C$ is \em well behaved \em if $C$ is a fiber of $\pi^{(\alpha)}$ 
or $C$ is the graph of a section of $\pi^{(\alpha)}$ such that $C\cap Z^\alpha_b=\emptyset$ and $C$ intersects $Z^\alpha_s$ at exactly one point with multiplicity $e_\alpha$.

We are now able to state the main theorem of this section.
We will prove it at the end of the section.

\begin{theorem}\label{SECONDSEQ}
Let $M$ be a germ of a complex analytic manifold at a point $o$.
Let $N$ be a smooth surface of $M$.
Let $S$ be a singular surface of $M$.
Let $\pi:\widetilde M\to M$ be the second sequence of blow-ups. 
Then $\Sigma_o^N(S) = \Sigma_o^N$ if and only if one of the following conditions is verified:

\begin{enumerate}[$(a)$]
\item
The tangent cone of $S$ is not a union of planes
\item
There is an integer $k$ and an irreducible component of $S^k\cap E^k$ that is not well behaved.
\item
There is a regular point $o_1$ of $ \widetilde N$ such that $\Sigma_{o_1}^{ \widetilde N}( \widetilde S)\supset \Sigma_{o_1}^{ \widetilde N}$.
\end{enumerate}
\end{theorem}

We have contact transformations 
\[
\pi_{k+1}:\mathbb P^*\langle M^{k+1}/N^{k+1}\rangle \to \mathbb P^*\langle M^k/N^k\rangle,
\]
$k\ge 0$, such that the diagrams

\[
\begin{array}{ccccc}
  \mathbb{P}^*\langle M^k/N^k\rangle  &  \overset{\pi_{k+1}}{\longleftarrow} &  \mathbb{P}^*\langle M^{k+1}/N^{k+1}\rangle    \\
 \pi_{M^k} \downarrow \qquad &  &   \qquad \downarrow \pi_{M^{k+1}} \\
  M^k  & \overset{\pi^{k+1}}{\longleftarrow} &   M^{k+1}
\end{array}
\]
commute. 
Set $\tau^k=\pi^2 \circ\cdots\circ\pi^k:M^k \to M^1$,
\[
\tau_k=\pi_2 \circ\cdots\circ\pi_k:P^*\langle M^k/N^k\rangle \to P^*\langle M^1/N^1\rangle,
\qquad k\ge 2.
\]
Let $\alpha\in I^k_b\cup I^k_s$. Let $i\in\{1,2,3,4\}$.
There is a system of  coordinates 
$(u_{\alpha,i},v_{\alpha,i},w_{\alpha,i};\xi_{\alpha,i}:\eta_{\alpha,i}:\zeta_{\alpha,i})$ 
on ${W}_{\alpha,i}=\pi_{M^k}^{-1}(U_{\alpha,i})$  such that 
\begin{equation}\label{1FORM}
\xi_{\alpha,i}\frac{du_{\alpha,i}}{u_{\alpha,i}}+
\eta_{\alpha,i}\frac{dv_{\alpha,i}}{v_{\alpha,i}}+
\zeta_{\alpha,i}dw_{\alpha,i}
\end{equation}
is the restriction to $W_{\alpha,i}$ of the canonical $1$-form of $T^*\langle M^k/N^k\rangle$.

\begin{lemma}\label{ADICAO}
If  $|\alpha|=k$ and $\alpha>1$, $\tau_k(W_{\alpha,i}\cap \pi_{M^k}^{-1}(E^k))\subset W_{1,1}\cap \pi_{M^1}^{-1}(Z^1)$, $i=1,2$ and $\tau_k(W_{\alpha,i}\cap \pi_{M^k}^{-1}(E^k))\subset W_{1,3}\cap \pi_{M^1}^{-1}(Z^1)$, $i=3,4$. 
Moreover, the restriction of $\tau_k$ to $W_{\alpha,i}\cap \pi_{M^k}^{-1}(E^k)$ is given by
\[
w_{1,1}=w_{\alpha,1}=w_{\alpha,2}, 
\:
w_{1,3}=w_{\alpha,3}=w_{\alpha,4},
\:
\zeta_{1,1}=\zeta_{\alpha,1}=\zeta_{\alpha,2},
\:
\zeta_{1,3}=\zeta_{\alpha,3}=\zeta_{\alpha,4};
\]
\[
\xi_{1,i}=n_{\al_\pi}\xi_{\alpha,i}-n_\alpha\eta_{\alpha,i}, 
\qquad
\eta_{1,i}=d_\alpha\eta_{\alpha,i}-d_{\alpha_\pi}\xi_{\alpha,i},
\qquad
\hbox{if }~i=1,3;
\]
\[
\xi_{1,i}=n_\alpha\eta_{\alpha,i}-n_{\al_\omega}\xi_{\alpha,i},
\qquad
\eta_{1,i}=d_{\alpha_\omega}\xi_{\alpha,i}-d_{\alpha}\eta_{\alpha,i},
\qquad
\hbox{if }~i=2,4.
\]
\end{lemma}
\begin{proof}

If $i=1,3$, the pull-back by $\pi_{k+1}|_{W_{\alpha_b,i}}$
of (\ref{1FORM}) equals
\[
(\xi_{\alpha,i}+\eta_{\alpha,i})\frac{du_{\alpha_b,i}}{u_{\alpha_b,i}}+
\eta_{\alpha,i}\frac{dv_{\alpha_b,i}}{v_{\alpha_b,i}}+
\zeta_{\alpha_b,i}dw_{\alpha_b,i}.
\]
Hence $\pi_{k+1}|_{W_{\alpha_b,i}\cap \pi_{M^k}^{-1}(E^k)}$ is given by the relations 
\[
w_{\alpha,i}=w_{\alpha_b,i}
\qquad
\xi_{\alpha,i}=\xi_{\alpha_b,i}-\eta_{\alpha_b,i},
\qquad
\eta_{\alpha,i}=\eta_{\alpha_b,i},
\qquad
\zeta_{\alpha,i}=\zeta_{\alpha_b,i}.
\]
Therefore the restriction of $\tau_{k+1}$ to $W_{\alpha_b,i}\cap \pi_{M^{k+1}}^{-1}(E^{k+1})$ is given by
\[
\xi_{1,i}=
n_{\alpha_\pi}\xi_{\alpha,i}-n_\alpha\eta_{\alpha,i}=
n_{\alpha_\pi}\xi_{\alpha_b,i}-(n_\alpha+n_{\alpha_\pi})\eta_{\alpha_b,i}=
n_{{\alpha_b}_\pi}\xi_{\alpha_b,i}-n_{\alpha_b}\eta_{\alpha_b,i},
\]
\[
\eta_{1,i}=d_\alpha\eta_{\alpha,i}-d_{\alpha_\pi}\xi_{\alpha,i}=
(d_\alpha+d_{\alpha_\pi})\eta_{\alpha_b,i}+d_{\alpha_\pi}\xi_{\alpha_b,i}=
d_{\alpha_b}\eta_{\alpha_b,i}+d_{{\alpha_b}_\pi}\xi_{\alpha_b,i}
\]
and $w_{1,i}=w_{\alpha_b,i}$, $\zeta_{1,i}=\zeta_{\alpha_b,i}$.
\end{proof}

\noindent
There is a canonical embedding of $\mathbb P^*\langle E^k/Z^k\rangle$ into  $\mathbb P^*\langle M^k/N^k\rangle$.
Moreover,
\[
\pi^k(E^k) \subset Z^{k-1}, \qquad 
\pi_k(\mathbb P^*\langle E^k/Z^k\rangle) \subset \mathbb P^*\langle M^{k-1}/N^{k-1}\rangle \cap \pi_{M^{k-1}}^{-1}(Z^{k-1}).
\]
Hence $\tau_k(\mathbb P^*\langle E^k/Z^k\rangle) \subset \mathbb P^*\langle M^{1}/N^{1}\rangle \cap \pi_{M^1} ^{-1}(Z^1)$.
Therefore $\pi_0\circ\pi_1\circ\tau_k$ defines a map $\upsilon_k:\mathbb P^*\langle E^k/Z^k \rangle \to \Sigma^N_o$.

Set $U'_{\alpha,i}=E^k\cap U_{\alpha,i}$, $W'_{\alpha,i}=\mathbb P^*\langle E^k/Z^k\rangle \cap W_{\alpha,i}$. Notice that
\[
W'_{\alpha,i}=
\{ (u_{\alpha,i},v_{\alpha,i},w_{\alpha,i};\xi_{\alpha,i}:\eta_{\alpha,i}:\zeta_{\alpha,i}) 
: u_{\alpha,i}=\xi_{\alpha,i}=0  \}
\]
and 
$U'_{\alpha,i}=
\{ (u_{\alpha,i},v_{\alpha,i},w_{\alpha,i}) 
: u_{\alpha,i}=0  \}$.
Moreover,
\[
\eta_{\alpha,i}\frac{dv_{\alpha,i}}{v_{\alpha,i}}+
\zeta_{\alpha,i}dw_{\alpha,i}
\]
is the restriction to $W'_{\alpha,i}$ of the canonical $1$-form of $T^*\langle E^k/Z^k\rangle$.

\begin{lemma}
If  $|\alpha|=k$,   $\upsilon_k$ is given by
\[
\eta=(-1)^{i+1}e_\alpha\eta_{\alpha,i}-w_{\alpha,i}\zeta_{\alpha,i},
\qquad
\zeta=\zeta_{\alpha,i},
\qquad
\hbox{if } ~i=1,2,
\]
\[
\zeta=(-1)^{i+1}e_\alpha\eta_{\alpha,i}-w_{\alpha,i}\zeta_{\alpha,i},
\qquad
\eta=\zeta_{\alpha,i},
\qquad
\hbox{if } ~i=3,4.
\]
\end{lemma}
\begin{proof}
Assume $\alpha>1$.
Let
\[
\xi\frac{dx}{x}+\eta dy+\zeta dz, 
\quad
\xi_2\frac{dx_2}{x_2}+\eta_2 \frac{dy_2}{y_2}+\zeta_2 dz_2,
\quad
\xi_3\frac{dx_3}{x_3}+\eta_3 \frac{dy_3}{y_3}+\zeta_3 dz_3
\]
be the canonical $1$-form of $T^*\langle M/N\rangle$,
the restriction to $\pi_{M^0}^{-1}(V_i)$ of the canonical $1$-form of
$T^*\langle M^0/N^0\rangle$, $i=2,3$.
By Lemma \ref{ADICAO} the restriction of $\tau_k$ to $W'_{\alpha,i}$ is given by
\[
w_{1,i}=w_{\alpha,i},
\qquad
\xi_{1,i}=(-1)^{i}n_\alpha\eta_{\alpha,i},
\qquad
\eta_{1,i}=(-1)^{i+1}d_\alpha\eta_{\alpha,i},
\qquad
\zeta_{1,i}=\zeta_{\alpha,i},
\]
i=1,2,3,4. The result follows from the fact that the restriction of $\pi_0$ to $\mathbb P^*\langle M^{0}/N^{0}\rangle \cap \pi_{M^0} ^{-1}(E^0)$ is given by
\[
\eta=\eta_2-\xi_2-z_2\zeta_2, ~~~
\zeta=\zeta_2;
\qquad
\eta=\zeta_3, ~~~
\zeta=\eta_3-\xi_3-z_3\zeta_3, 
\]
and $\pi_1$ is given by
\[
z_2=w_{1,1},
\xi_2=\xi_{1,1}-\eta_{1,1},
\eta_2=\eta_{1,1},
\zeta_2=\zeta_{1,1};
\]
\[
z_2=w_{1,2}, 
\xi_2=\eta_{1,2}, 
\eta_2=\xi_{1,2}-\eta_{1,2}, 
\zeta_2=\zeta_{1,2};
\]
\[
z_3=w_{1,3},
\xi_3=\xi_{1,3}-\eta_{1,3},
\eta_3=\eta_{1,3}, 
\zeta_3=\zeta_{1,3};
\]
\[
z_3=w_{1,4},
\xi_3=\eta_{1,4},
\eta_3=\xi_{1,4}-\eta_{1,4}, 
\zeta_3=\zeta_{1,4}.
\]
The proof in the case $\alpha<1$ is similar.
Remark that $e_{\alpha^{-1}}=e_\alpha$.
\end{proof}

\begin{lemma}\label{CURVE}
For each $\alpha$ and each curve $C$ of $E^{(\alpha)}$,
$C$ intersects $Z^\alpha$.
\end{lemma}

\begin{proof}
Assume that $C$ does not intersect  $Z^\alpha$.
The intersection of $C$ with $U'_{\alpha,3}$ is defined by a polynomial
$\sum_{i=0}^\ell a_i(w_{\alpha,3})v_{\alpha,3}^i$. Hence $a_0\in\mathbb C^*$.

There is an integer $\mu\ge 0$ such that $C\cap U'_{\alpha,1}$ is given by the polynomial
\[
w_{\alpha,1}^\mu (\sum_{i=0}^\ell a_i(w_{\alpha,1}^{-1})w_{\alpha,1}^{e_\alpha i}v_{\alpha,1}^i).
\]
Since $C$ does not intersect  $Z^\alpha$, $\mu=0$.
The intersection of $C$ with $U'_{\alpha,2}$ is defined by 
\[
\sum_{i=0}^\ell a_i(w_{\alpha,2}^{-1})w_{\alpha,2}^{e_\alpha i}v_{\alpha,2}^{\ell-i}.
\]  
Hence there is $\lambda\in\mathbb C^*$ such that $a_\ell(t)=\lambda t^{e_\alpha \ell}$.
Finally, $C\cap U'_{\alpha,4}$ is given by
$\sum_{i=0}^\ell a_i(w_{\alpha,4})v_{\alpha,4}^{\ell-i}$. Therefore $a_\ell\in\mathbb C^*$, which leads to a contradiction.
\end{proof}

\begin{lemma}\label{WELLB}
Let $C$ be an irreducible curve of $E^{(\alpha)}$. 
The image by $\upsilon_k$ of $\Gamma=\mathbb P^*_C\langle E^k/N^k\rangle$ is different from
$\Sigma_o^N$ if and only if $C$ is well behaved
\end{lemma}
\begin{proof}
Assume $\upsilon_k(\Gamma)$ is different from $\Sigma_o^N$.
Set 

\noindent
$p=-\eta\zeta^{-1}, ~  p_i=-\eta_{\alpha,i}\zeta_{\alpha,i}^{-1}$, 
if  $i=1,2$; 

\noindent
$p=-\zeta\eta^{-1}, ~  p_i=-\eta_{\alpha,i}\zeta_{\alpha,i}^{-1}$, 
if $i=3,4$.

\noindent
The restriction of $\upsilon_k$ to $W'_{\alpha,i}$ is given by
\[
p=(-1)^{i+1}e_\alpha p_i+w_{\alpha,i}.
\] 

By Lemma \ref{CURVE}, $C$ intersects $Z^k$ at a point $o_1$. 
Let $C_1$ be a branch of the germ of $C$ at $o_1$.
Then  $C_1=\{ w_{\alpha,i}=c\}$ or $C_1$ admits one of the following parametrizations

\begin{equation}\label{PARA}
v_{\alpha,i}=t^a, 
\quad
w_{\alpha,i}=c+t^b\varepsilon;
\qquad
v_{\alpha,i}=t^b\varepsilon,
\quad
w_{\alpha,i}=c+t^a;
\end{equation}

where $b\geq a$ and $\varepsilon$ is a unit of $\mathbb C\{t\}$. 
In the first case $C_1$ is well behaved.
Assume $C_1$ admits the first parametrization.
Setting $\Gamma_1=\mathbb P^*_{C_1}\langle E^k/Z^k\rangle$,
 $\Gamma_1$ admits a local parametrization given by (\ref{PARA}) and 
$p_i=(t^b/a)(b\varepsilon+t\varepsilon')$.
Therefore $\upsilon_k(\Gamma_1)$ contains the set of points $p$ such that
\[
p=c+t^b[((-1)^{i+1}e_\alpha (b/a)+1)\varepsilon+(-1)^{i+1}e_\alpha (1/a) t\varepsilon']
\]
and $|t|<<1$.

This set is finite if and only if $\varepsilon$ is the solution of an ODE $t\varepsilon'+\lambda\varepsilon=0$.
Since $\varepsilon$ is a unit, $\lambda=0$. Hence 
$a=(-1)^{i}be_\alpha$.
Since $a,b,e_\alpha$ are positive, $i$ is even. 
Hence $C$ cannot intersect $Z^\alpha_s$. Moreover,
$C\cap U'_{\alpha,i}$ is described by an equation of the type
\begin{equation}\label{SECTION}
v_{\alpha,i}=(\mu w_{\alpha,i}+\nu)^{e_\alpha},
\end{equation}
 Hence $C_1$ is the graph of a section of $\pi^{(\alpha)}$.
 The remaining case can be treated in a similar way.
Remark that in each case $C=C_1$.

Let $C$ be a section of $\pi^{(\alpha)}$ verifying the statements of the lemma.
Then $C$ is a section of the restriction  $\pi^{[\alpha]}$ of $\pi^{(\alpha)}$ to $E^{(\alpha)}\setminus Z^\alpha_s$.
Since $\pi^{[\alpha]}$ is a line bundle of degree $e_\alpha$ and $C$ has  a zero of order $e_\alpha$, $C$  is of the type (\ref{SECTION}).
\end{proof}

\begin{proof}[Proof of Theorem \ref{SECONDSEQ}]
If $(a)$ $[(b),(c)]$ holds it follows from Lemma \ref{NONTRIVCONE} 
[Lemma \ref{WELLB}, Theorems \ref{SIGMAAB} and \ref{TRIVCONE}]
that $\Sigma_o^N(S)=\Sigma_o^N$.

Let $\pi_k:M_k\to M$ be the second sequence of blow ups.
Let $\ell\leq k$ and let $F_\ell$ be an irreducible component of $N_\ell$.
Let $F'_\ell$ be the intersection of $F_\ell$ with the regular part of $N_\ell$.

Assume $(a),(b),(c)$ do not hold.
Since $(c)$ does not hold, the closure of 
\[
\mathbb P^*_{S_k}\langle M_k/N_k\rangle\cap 
\pi_{M_k}^{-1}(F'_k)
\]
is the closure of $\mathbb P^*_{S_k\cap F'_k}F'_k$.
We show by induction in $\ell$, using theorems \ref{SIGMAAB} and \ref{TRIVCONE} that
\[
\mathbb P^*_{S_{k-\ell}}\langle M_{k-\ell}/N_{k-\ell}\rangle\cap 
\pi_{M_{k-\ell}}^{-1}(F_{k-\ell})
\]
is the closure of $\mathbb P^*_{S_{k-\ell}\cap F'_{k-\ell}}F'_{k-\ell}$,
for each $\ell\le k$.

Since $(a)$ does not hold, it follows from Theorem \ref{TRIVCONE} that $\pi_1(\mathbb P^*_{S_1\cap F_1}F_1)$ is finite.
Since $(b)$ does not hold, it follows from Lemma \ref{WELLB} that $\upsilon_\ell(\mathbb P^*_{S_\ell\cap F_\ell}F_\ell)$ is finite, 
for  $2\le \ell \le k$.
\end{proof}

\section{Main Results}\label{MRE}

\begin{blank}\label{MAINALGORITHM}\em
Let $N_0$ be a smooth surface of a germ of a manifold $M_0$ of
dimension $3$ at a point $o$.
Let $S_0$ be a surface of $M_0$ that does not contain $N_0$.

Let $N_k$ be a normal crossings divisor of a manifold $M_k$ of dimension $3$.
Let $S_k$ be a singular surface of $M_k$ that does not contains any irreducible component of $N_k$.
Assume Sing$^{N_k}(S_k)$ does not intersect the singular locus of $N_k$.

Let $\sigma\in$Sing$^{N_k}(S_k)\cap N_k$.
Assume the germ of Sing$^{N_k}(S_k)\cap N_k$
at $\sigma$ is not smooth or is not transversal to $N_k$.
If $S_k$ is non degenerated at $\sigma$, we blow up $M_k$ at $\sigma$ followed by the first sequence of blow ups.
Otherwise we perform the second sequence of blow ups at $\sigma$.
After modifying $M_k$ at each point of the finite set Sing$^{N_k}(S_k)\cap N_k$,
we obtain a map $\pi_{k+1}:M_{k+1}\to M_k$.
Set $N_{k+1}=\pi_{k+1}^{-1}(N_k)$.
Let $S_{k+1}$ be the strict transform of $S_k$ by $\pi_{k+1}$.
Applying, accordingly,  the first sequence or the second sequence of blow ups, we guarantee that 
 Sing$^{N_{k+1}}(S_{k+1})$ does not intersect the singular locus of $N_{k+1}$.
\end{blank}

\begin{lemma}\label{DESING}
There  is an integer $k$ such that each connected component
 of Sing$^{N_k}(S_k)$ is a smooth curve transversal to $N_k$. 
 Hence the procedure described in paragraph \ref{MAINALGORITHM} 
 will terminate after a finite number of steps.
\end{lemma}

\begin{proof}
Notice that, for each $\ell$,
Sing$^{N_{\ell+1}}(S_{\ell+1})$ is the strict transform by
$\pi_{\ell+1}$ of Sing$^{N_{\ell}}(S_{\ell})$.

Let $C$ be an irreducible singular curve of a germ of manifold 
$M$ of dimension $3$ at a point $o$.
Let $\gamma:(\mathbb C,0)\to C$ be the normalization of $C$.
Let $\Gamma$ be the semi group of the orders of the functions $\gamma^*f$,
$f\in \mathcal O_{M,o}$.
Let $m_C$ be the smallest positive integer that belongs to $\Gamma$.
The integer $m_C$ equals the multiplicity of $C$.
Let $n_C$ be the infimum of $\Gamma\setminus (m_C)$.

Let $\widetilde C$ be the strict transform of $C$ by
the blow up of $M$ along a smooth line that contains $o$.
Then
\[
m_{\widetilde C} < m_C 
\qquad
\hbox{or}
\qquad
m_{\widetilde C} = m_C 
\hbox{ and }
n_{\widetilde C}\le n_C.
\]
Hence the invariant does not get worse. Let $\widetilde C$ be the strict transform of $C$ by
the blow up of $M$ along $o$.
Then
\[
m_{\widetilde C} < m_C 
\qquad
\hbox{or}
\qquad
m_{\widetilde C} = m_C 
\hbox{ and }
n_{\widetilde C} < n_C.
\]
Hence the invariant improves. The facts above show that there is an integer $k$ such that
Sing$^{N_k}(S_k)$ is a union of smooth curves.
Hence there is an integer $\ell$ such that
Sing$^{N_\ell}(S_\ell)$ is a union of smooth curves transversal to $N_\ell$.

Let $C,C'$ be two curves of $M$.
Let $m(C,C')$ be the number of blow ups necessary to separate $C$ and $C'$.
Let $\widetilde C\; [\widetilde C']$ be the strict transform of $C[C']$ by
the blow up of $M$ along a smooth line that contains $o$.
Then
\[
m(\widetilde C, \widetilde C') \le  m( C,  C').
\]
Let $\widetilde C\; [\widetilde C']$ be the strict transform of $C[C']$ by
the blow up of $M$ along  $o$.
Then
\[
m(\widetilde C, \widetilde C') <  m( C,  C').
\]
Hence there is an integer $m$ such that each connected component of
Sing$^{N_m}(S_m)$ is a  smooth curve transversal to $N_m$.
\end{proof}

\begin{theorem}\label{DECIDE}
Let $M_0\leftarrow M_1 \leftarrow \cdots \leftarrow M_k$ be the sequence 
of morphisms described in paragraph \ref{MAINALGORITHM}. 
Then $\Sigma^N_o(S)=\Sigma^N_o$ if and only if one of the following statements holds.
\begin{enumerate}[$(a)$]
\item
somewhere along the process a curve that is not well behaved is produced,
\item there is 
$\sigma\in (S_k\cap N_k)\setminus$\em Sing\em$^{N_k}(S_k)$
such that $\sigma\in \Xi^N_\rho(S)$
for some projection $\rho$ compatible with $N$,
\item
there is 
$\sigma\in S_k\cap N_k\cap $\em Sing\em$^{N_k}(S_k)$
such that $\sigma\in \overline{\Xi^N_\rho(S)\setminus \textrm{Sing}^N(S)}$
for some projection $\rho$ compatible with $N$.
\end{enumerate}
\end{theorem}

\begin{proof}
Assume $(a)$ holds. Then there is an integer $\ell$ such that a 
non well behaved curve is produced along the second sequence of blow ups
$M_\ell\leftarrow M_{\ell+1}$.  
By Theorem \ref{SECONDSEQ} there is $\sigma_\ell\in S_\ell\cap N_\ell$ such that 
$\Sigma_{\sigma_\ell}^{N_\ell}(S_\ell)=\Sigma_{\sigma_\ell}^{N_\ell}$.
We prove by induction in $n$, using Theorem \ref{SECONDSEQ}, that for each $n\le\ell$
there is 
$\sigma_{\ell-n}\in S_{\ell-n}\cap N_{\ell-n}$ such that 
$\Sigma_{\sigma_{\ell-n}}^{N_{\ell-n}}(S_{\ell-n})=\Sigma_{\sigma_{\ell-n}}^{N_{\ell-n}}$.

Assume $(b)$ $[(c)]$ holds. By Theorem \ref{CHANGESMOOTH} [Theorem \ref{CHANGETRANS}] 
there is $\sigma_k\in S_k\cap N_k$ such that 
$\Sigma_{\sigma_k}^{N_k}(S_k)=\Sigma_{\sigma_k}^{N_k}$.
We repeat the argument of the previous paragraph.

Assume $(a),(b),(c)$ do not hold.
Since $(b),(c)$ do not hold,
 $\Sigma_{\sigma}^{N_k}(S_k)$ is finite for each $\sigma\in S_k\cap N_k$.
We can now show by induction in $\ell$, using Theorem \ref{SECONDSEQ} and the fact that $(a)$ does not hold,
that $\Sigma_{\sigma}^{N_{k-\ell}}(S_{k-\ell})$is finite for each $\sigma\in S_{k-\ell}\cap N_{k-\ell}$ and each $\ell\le k$.
\end{proof}

\begin{theorem}\label{JUNG}
Let $S$ be a surface of the germ of a complex manifold $M$ at a point $o$.
Assume $\Sigma_o(S)$ is finite.
Let us blow up $M$ at $o$.
Let us apply the procedure described in paragraph \ref{MAINALGORITHM} at each singular point $\sigma$ 
of the strict transform $S_0$ of $S$ that belongs to the exceptional divisor of the blow up.
We obtain is this way a manifold $M_n$, 
a normal crossings divisor $N_n$ and a surface $S_n$ such that at each point $\sigma$
of $S_n\cap N_n$, the germ of $S_n$ at $\sigma$ is a 
quasi ordinary singularity relative to a projection $\rho$ compatible with $N_n$. 
Moreover, $\Delta_\rho S_n\cup \rho(N_n)$ is a normal crossings divisor at $\rho(o)$.
\end{theorem}

\begin{proof}
By Lemma \ref{DESING}, Sing$^{N_k}(S_k)$ is smooth and transversal to $N_k$ at smooth points of $N_k$.

Let $\sigma\in S_k\cap N^{\sigma}_k$. The procedure of paragraph \ref{MAINALGORITHM} relies on the procedure of paragraph \ref{PRIMEIRA}. Therefore $\Delta_{\rho_k}(S_k)\subset \rho_k(N_k)$ for each projection $\rho_k$ compatible with the germ of $N_k$ at $\sigma$. Therefore $S_k$ is quasi ordinary at the singular points of $N_k$.

By Theorem  \ref{DECIDE}, $\Sigma_\sigma^{N_k}(S_k)$ is finite for each $\sigma\in S_k\cap N_k$.
By Theorem \ref{CHANGESMOOTH}, $S_k$ is quasi ordinary at the regular points of $N_k$ that do not belong to Sing$^{N_k}(S_k)$.
By Theorem \ref{CHANGETRANS}, $S_k$ is quasi ordinary at the points of Sing$^{N_k}(S_k)\cap N_k$.
\end{proof}

Let $M$ be an affine chart with coordinates $(x,y,z)$.
Let $\rho:M\to X$ be the linear projection $(x,y,z)\mapsto (x,y)$.
Set $o=(x,y,z)$, $\sigma=(0,0)$. Let $S$ be a surface of $M$.
Let $\pi_1:M_1\to M$ be the blow up of $M$ at $o$.
Set $E_1=N_1=\pi_1^{-1}(o)$.
Let $S_1$ be the strict transform of $S$ by $\pi_1$.

\begin{example}\em
Let $S$ be the surface defined by the polynomial
\[
z^2-x^2(x+y^2).
\]
Notice that $C_o(S)=\{z=0\}$ and
\[
C_\sigma(\Delta_z(S))=
C_\sigma(\rho(\hbox{Sing}(S)))=\{x=0\}.
\]
By Theorem 1.4.4.1 of \cite{AMERICAN}, $(0:0:1)\in\Sigma_o(S)\subset \Sigma_{\ell}$,
where 
\[
\ell=\{x=z=0\}
\qquad
\hbox{and}
\qquad 
\Sigma_{\ell}=\{\eta=0\}.
\]

Set $x=x_1y_1,y=y_1,z=y_1z_1$. 
Then $o_{\ell}$ is the origin of the chart $W_1$ with coordinates $(x_1,y_1,z_1)$.
Moreover, $N_1\cap W_1=\{y_1=0\}$ and $S_1\cap W_1$ is defined by the polynomial
\[
z_1^2-x_1^2y_1(x_1+y_1).
\]
Since $C_{o_\ell}(S_1)=\{z_1=0\}$, $S_1$ is non degenerated at $o_\ell$.

Let $\pi_2:M_2\to M_1$ be the blow up of $M_1$ at $o_\ell$.

Set $x_1=x_2y_2,y_1=y_2,z_1=y_2z_2$.
If $W_2$ is the affine chart with coordinates $(x_2,y_2,z_2)$,
$N_2\cap W_2=\{y_2=0\}$ and $S_2\cap W_2$ is defined by the polynomial
\[
z_2^2-x_2^2y_2^2(x_2+1).
\]
Since $\Delta_{z_2}(S_2)=\{x_2y_2(x_2+1)  \}$,
$\Sigma_{(-1,0,0)}^{N_2}(S_2)=\{\eta=0\}$.
Hence $\Sigma_o(S) = \Sigma_{\ell}$.
\end{example}

\begin{example}\em
Let $S$ be the swallowtail surface, defined by the polynomial
\[
256z^3-27y^4-128x^2z^2+144xy^2z+16x^4z-4x^3y^2.
\]
Notice that $C_o(S)=\{z=0\}$ and
\[
C_\sigma(\Delta_z(S))=
C_\sigma(\rho(\hbox{Sing}(S)))=\{y=0\}.
\]
By Theorem 1.4.4.1 of \cite{AMERICAN}, $(0:0:1)\in\Sigma_o(S)\subset \Sigma_{\ell}$,
where 
\[
\Sigma_{\ell}=\{\xi=0\}.
\]

Set $x=x_1,y=x_1y_1,z=x_1z_1$. 
Then $o_{\ell}$ is the origin of the chart $W_1$ with coordinates $(x_1,y_1,z_1)$.
Moreover, $N_1\cap W_1=\{x_1=0\}$ and $S_1\cap W_1$ is defined by the polynomial
\[
256z_1^3-27x_1y_1^4-128x_1z_1^2+144x_1y_1^2z_1+16x_1^2z_1-4x_1^2y_1^2.
\]
Since $C_{o_{\ell}}(S_1)=\{ z=0 \}$, $S_1$ is non degenerated at $o_{\ell}$.

Let $\pi_2:M_2\to M_1$ be the blow up of $M_1$ at $o_{\ell}$.

Set $x_1=x_2,y_1=x_2y_2,z_1=x_2z_2$.
If $W_2$ is the chart with coordinates $(x_2,y_2,z_2)$, $N_2\cap W_2=\{x_2=0\}$
and $S_2\cap W_2$ is defined by the polynomial
\[
256z_2^3-27x_2^2y_2^4-128z_2^2+144x_2y_2^2z_2+16z_2-4x_2y_2^2.
\]
Since $S_2\cap N_2\cap W_2=\{x_2=z_2(4z_2-1)=0\}$ and $\Delta_{z_2}(S_2)=\{ x_2y_2(27x_2y_2^2+8)=0 \}$, 
we need to analyse the points $o_1=(0,0,0)$ and $o_2=(0,0,1/4)$.
Since Sing$^{N_2}(S_2)$ equals $\{y_2=4z_2-1=0\}$ in a neighbourhood of $N_2$, 
it follows from Theorem \ref{CHANGESMOOTH} that $\Sigma_{o_1}^{N_2}(S_2)$ is finite.
Since $m_{(a,0,1/4)}(S_2)$ does not depend on $a$, for $|a|<<1$, 
it follows from Theorem \ref{CHANGETRANS} that $\Sigma_{o_2}^{N_2}(S_2)$ is finite.

Set $x_1=x_3y_3,y_1=y_3,z_1=x_3z_3$.
If $W_3$ is the chart with coordinates $(x_3,y_3,z_3)$, $N_2\cap W_3=\{x_3y_3=0\}$ and $S_2\cap W_3$ is  defined by the polynomial
\[
256z_3^3-27x_3y_3^2-128x_3z_3^2+144x_3y_3z_3+16x_3^2z_3-4x_3^2y_3.
\]
Set $o_3=(0,0,0)$. Since the intersection of $S_2$ with the singular locus $N^\sigma_2$ of $N_2$ equals
$\{o_3\}$, we need to compute $\Sigma^{N_2}_{o_3}(S_2)$.

Let $\pi_3:M_3\to M_2$ be the blow up of $M_2$ along $N^\sigma_2$.

Set $x_3=x_4y_4,y_3=y_4,z_3=z_4$.
If $W_4$ is the chart with coordinates $(x_4,y_4,z_4)$, 
$E_3\cap W_4=\{y_4=0\}$, $N_3\cap W_4=\{x_4y_4=0\}$ and
$S_3\cap W_4$ is defined by the polynomial
\[
256z_4^3-27x_4y_4^3-128x_4y_4z_4^2+144x_4y_4^2z_4+16x_4^2y_4^2z_4-4x_4^2y_4^3.
\]
Since $S_3\cap E_3\cap W_4=\{z_4=0\}$ and $\Delta_{z_4}(S_3)=\{x_4y_4(8x_4-27)=0\}$,
we need to analyse the points $o_4=(0,0,0)$ and $o_5=(-27/8,0,0)$.
By Lemma \ref{TRIVIAL}, $\Sigma_{o_4}^{N_3}(S_3)$ is finite.

Set $x_5=x_4-27/8$, $y_5=y_4$, $z_5=z_4+9y_4/32+x_4y_4/36$. 
There are integers $a_1,...,a_5$ such that $S_3$ is defined near $o_5$ by the polynomial
\[
a_1z_5^3+a_2 y_5 z_5^2+a_3x_5^3y_5^3+a_4x_5^2y_5^2z_5+a_5x_5y_5z_5^2.
\]
Now Sing$^{N_3}(S_3)=\{y_5=z_5=0\}$ and it is easy to check that 
$m_\sigma(S_3)$ does not depend on $\sigma\in$Sing$^{N_3}(S_3)$.
By Theorem \ref{CHANGETRANS}, $\Sigma_{o_5}^{N_3}(S_3)$ is finite. 
By Theorem \ref{DECIDE}, $\Sigma_{o}^{N}(S)=\{(0:0:1)\}$.
\end{example}

\begin{example}\em
Set $N=\{x=0\}$, $o=(0,0,0)$ and $\sigma=(0,0)$.
Let $S$ be the surface defined by the polynomial
\[
z^5-x^2y.
\]
The surface $S$ is degenerated at $o$. 
We will perform the second sequence of blow ups at $o$.
Let $\pi_1:M_1\to M$ be the blow up of $M$ at $o$.

Set $x=x_1y_1,y=y_1,z=y_1z_1$. 
If $W_1$ is the affine open set of $M_1$ with coordinates $(x_1,y_1,z_1)$,
 $S_1\cap W_1$ is defined by the polynomial
\[
x_1^2-y_1^2z_1^5
\]
and $S_1\cap E_1 \cap W_1=N_1^\sigma\cap W_1=\{x_1=y_1=0\}$.

Set $x=x_2z_2,y=y_2z_2,z=z_2$. 
If $W_2$ is the affine open set of $M_1$ with coordinates $(x_2,y_2,z_2)$,
 $S_1\cap W_2$ is defined by the polynomial
\[
z_2^2-x_2^2y_2,
\]
$S_1\cap E_1 \cap W_2=\{z_2=x_2y_2=0\}$ and $N_1^\sigma\cap W_2=\{x_2=z_2=0\}$.

The surface $S_1$ is smooth in the other chart.

Let $\pi_2:M_2\to M_1$ be the blow up of $M_1$ along $N_1^\sigma$.
Set $E_2=\pi_2^{-1}(N_1^\sigma)$.
Notice that $\pi=\pi_1\circ\pi_2:M_2\to M_1$ is the second sequence of blow ups of $S$.
Since $S_2$ is quasi ordinary at each point it is quite easy to verify that 
$\Sigma_\sigma^{N_2}(S_2)$ is finite for each $\sigma\in S_2$.

It follows from Theorem \ref{DECIDE} that $\Sigma_o(S)=\Sigma_o^N$ because the curve
$S_2\cap E_2$ has a singularity, hence one of its irreducible components is not well behaved.
\end{example}

\noindent
This work was partially supported by the Funda\c{c}\~{a}o para a Ci\^{e}ncia e a Tecnologia (Portuguese Foundation for Science and Technology) through the projects UID/MAT/04561/2013​ (CMAF-CIU) and UID/MAT/00297/2013 (Centro de Matem\'atica e Aplica\c{c}\~{o}es).

\end{document}